\documentclass{article}

\usepackage{CAAI_preprint}
\usepackage{multirow}

\title{\textbf{Stability analysis of a modified Leslie--Gower predation model with weak Allee effect in the prey}}

\author{
Claudio Arancibia--Ibarra \\
School of Mathematical Sciences, Queensland University of Technology, Brisbane, Australia\\
Facultad de Ingeniería y Negocios, Universidad de Las Américas, Santiago, Chile \\
 \texttt{claudio.arancibiaibarra@qut.edu.au} \\
  \AND
Jos\'e Flores \\
Department of Mathematics, The University of South Dakota, South Dakota, USA\\
 \texttt{Jose.Flores@usd.edu} \\
 \AND
Peter van~Heijster \\
Biometris, Wageningen University and Research, Wageningen, Netherlands\\
 \texttt{peter.vanheijster@wur.nl} \\  
}

\begin{document}
\maketitle
\begin{abstract} \normalsize

In this manuscript, we study a Leslie--Gower predator-prey model with a hyperbolic functional response and weak Allee effect. The results reveal that the model supports coexistence and oscillation of both predator and prey populations. We also identify regions in the parameter space in which different kinds of bifurcations, such as saddle-node bifurcations, Hopf bifurcations and Bogdanov--Takens bifurcations.
\end{abstract}


\section{Introduction}

Predator-prey models are studied in both applied mathematics~\cite{arancibia7,kundu,martinez2} and ecology~\cite{mateo,mondal,santos,turchin}. The goal of these studies is to describe and analyse the predation interaction between the predator and the prey and predict how they respond to future interventions~\cite{hooper,may}. These studies often use mathematical models to describe the species’ interactions and the time-series behaviours~\cite{moller,monclus}. The models aim to be representative of real natural phenomena capturing the essentials of the dynamics. However, new technologies used to study biological and physical phenomenon reveal that species’ interactions are more complex than previously used in the models~\cite{turchin,hanski,hanski2,roux}. The importance of these more complex interactions are becoming increasingly apparent as research findings have shown that ecosystem dynamics depend on the particular nature of the interaction processes, such as the functional response and predation rate~\cite{santos,turchin,bimler,wood}.  

A standard approach for using models to understand ecological systems is to design a framework based on simple principles and compare species’ abundance that result from the predicting analysis of those models. However, this approach becomes more difficult when additional nuances to standard models are added, making them more complex and difficult to analyse. For instance, Graham and Lambin~\cite{graham} showed that field-vole (\textit{Microtus agrestis}) survival can be affected by reducing least weasel (\textit{Mustela nivalis}) predation. They also demonstrated that weasels were suppressed in summer and autumn, while the vole (\textit{Microtus agrestis}) population always declined to low density. However, the authors argued that the underlying model was too difficult to study due to a large number of parameters. 

The model used by Graham and Lambin in~\cite{graham} is a Leslie--Gower predator-prey model~\cite{leslie} and is given by 
\begin{equation}\label{eq02}
\begin{aligned}
\dfrac{dN}{dt} &=	 rN\left( 1-\dfrac{N}{K}\right) - \dfrac{qNP}{N+a}\,, \\
\dfrac{dP}{dt} &=	 sP\left( 1-\dfrac{P}{hN}\right)\,.
\end{aligned}  
\end{equation}
Here, $N\left(t\right)>0$ and $P\left(t\right)>0$ are used to represent the size of the prey and predator population at time $t$ respectively. The prey population grows logistically with carrying capacity $K$ and intrinsic growth rate $r$. The growth of the predator is also logistic, with intrinsic growth rate $s$, but the carrying capacity is prey dependent and $h$ is a measure of the quality of the prey as food for the predator. The functional response is Holling type II where $q$ is the maximum predation rate per capita and $a$ is half of the saturated response level~\cite{turchin}. 

In population dynamics, many ecological mechanisms are connected with individual cooperation such as strategies to hunt, collaboration in unfavourable abiotic conditions and reproduction~\cite{courchamp}. An example of such an ecological mechanism is the cooperative interaction observed in Kakapo \textit{(Strigops habroptilus)}. This interaction can be influenced by reducing the fecundity levels and the population size~\cite{courchamp2}. Another factor that affects the population dynamics is the species density. For instance, when the predator population size is low they have more resources and benefits. However, there are species that may suffer from a lack of conspecifics, which may be less likely to reproduce or survive in a small-sized population~\cite{stephens}. In these instances, the size of the population is important, as for a smaller size of biomass adaptability may be diminished~\cite{courchamp2}. When Allee and collaborators~\cite{allee} analysed the data of the false weevil ({\em Tribolium confusum}) they observed that the highest growth rates of their populations per capita were at intermediate densities~\cite{courchamp2}. The fact that they were lower in high densities was not surprising, as intraspecific competition is high. In contrast, when fewer males were present, females produced fewer eggs, which is not an obvious correlation for an insect. In this case, optimal egg production was thus achieved at intermediate densities. This effect is now referred to as the \emph{Allee effect}~\cite{berec,allee}. It is modelled by incorporating the factor $(N-m)$ into the logistic equation, with the idea that when the prey population $N$ decreases below the Allee threshold $m>0$ the prey rate of growth becomes negative, indicating signs of extinction. Therefore, the prey logistic growth $r\left(1-N/K\right)$ is replaced by $r\left(1-N/K\right)\left(N-m\right)$. For $0<m<K$, the per-capita growth rate of the prey population with the Allee effect included is negative, but increasing, for $N\in[0,m)$, and this is referred to as the \emph{strong Allee effect}.  When $m\leq0$, the per-capita growth rate is positive but increases at low prey population densities and this is referred to as the \emph{weak Allee effect}~\cite{berec,courchamp2}. With the Allee effect included, the Leslie--Gower predator-prey model~\eqref{eq02} of~\cite{graham} becomes
\begin{equation}
\begin{aligned} \label{eq03}
\dfrac{dN}{dt} &=	 rN\left( 1-\dfrac{N}{K}\right)\left(N-m\right) - \dfrac{qNP}{N+a}=N\cdot W\left(N,P\right)\,, \\
\dfrac{dP}{dt} &=	 sP\left( 1-\dfrac{P}{hN}\right)=P\cdot R\left(N,P\right)\,.
\end{aligned}
\end{equation}
The aim of this manuscript is to study the Leslie--Gower predator-prey model with weak Allee effect on prey, that is~\eqref{eq03} with $m<0$. By considering a weak Allee effect in the prey population we complement the results of  Ostfeld and Canhan~\cite{ostfeld}. The authors studied the stabilisation of a file-vole population which depends on the variation in reproductive rate and recruitment of the population, i.e. an Allee effect. The weak Allee effect can be also observed in the File-vole species as the survival rate for adults were delayed~\cite{ostfeld}. However, these phenomena were not considered by the authors as they used model~\eqref{eq02}.  
This manuscript also extends some of the results obtained by Arancibia--Ibarra and Gonz\'alez--Olivares~\cite{arancibia} and Gonz\'alez--Olivares {\em et al.\ }~\cite{gonzalez3} for a modified Leslie--Gower model with $m=0$, that is, with a specific type of weak Allee effect, see second row of Table~\ref{T1}. Leslie--Gower models with strong Allee effect and different type of functional responses have been extensively studied in~\cite{gonzalez6,tintinago,arancibia3}, see first row of Table~\ref{T1}. In these articles the authors showed that the system, for certain system parameters, supports the extinction of both species and it also supports the stabilisation of both populations over the time, i.e. coexistence. Moreover, system~\eqref{eq03} with $m<0$ complements the results of the Leslie--Gower model studied by Courchamp {\em et al.}~\cite{courchamp2}. The authors in~\cite{courchamp2} argue that the Allee effect in prey generally destabilises the dynamics between the prey and the predator. Moreover, the Allee effect can prevent the oscillation of both populations. However, in this manuscript, we find that the Leslie--Gower model with weak Allee effect supports the coexistence and oscillation of both populations. 

\begin{table}
\centering 
\begin{tabular}{llccc}
\hline
Allee Threshold & Positive Equilibria & \multicolumn{3}{l}{Phase Plane}  \\
\multirow{2}{*}{$m>0$} & \multirow{2}{*}{At most 2} & & & \\
 && \includegraphics[width=3.5cm]{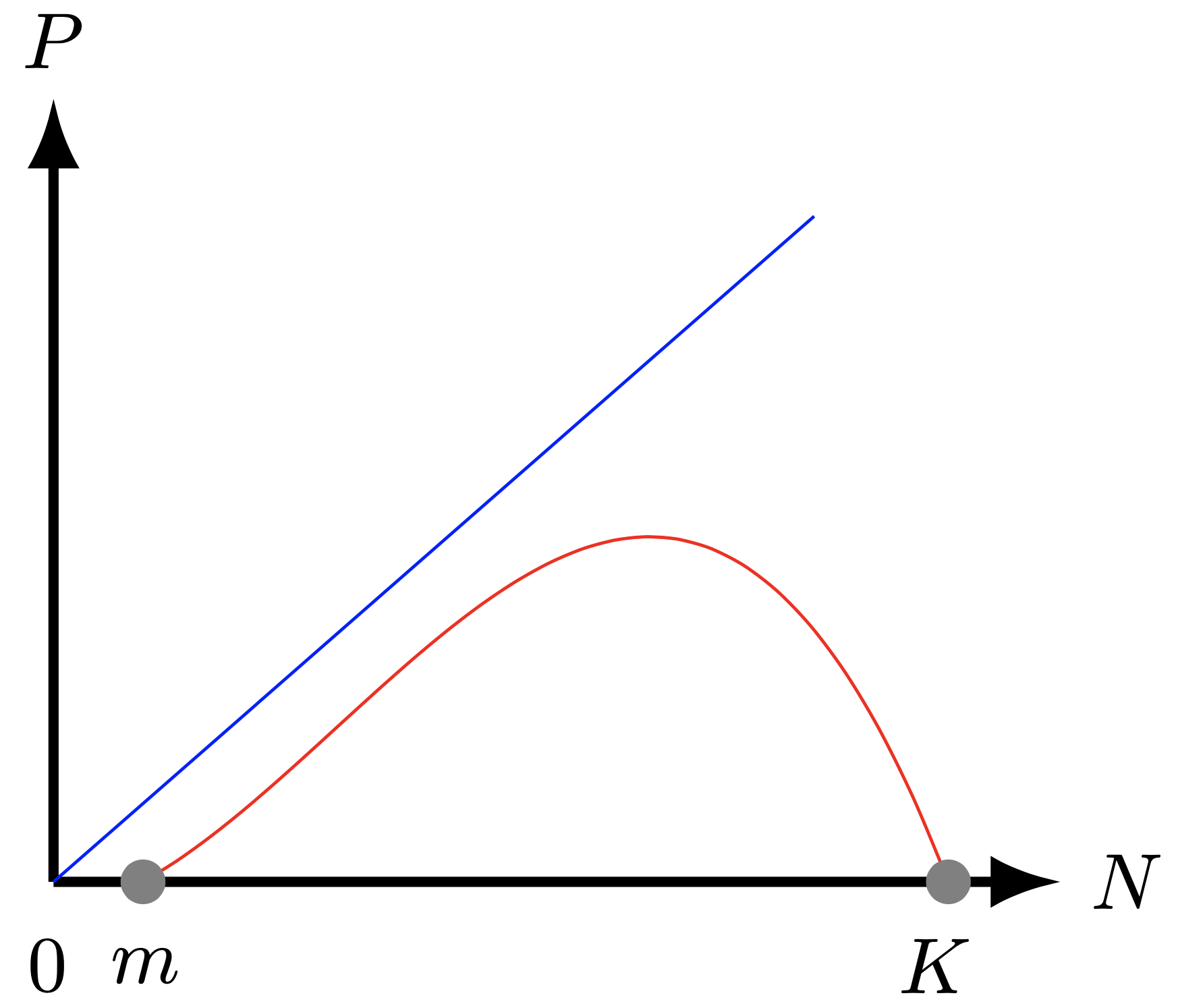} & \includegraphics[width=3.5cm]{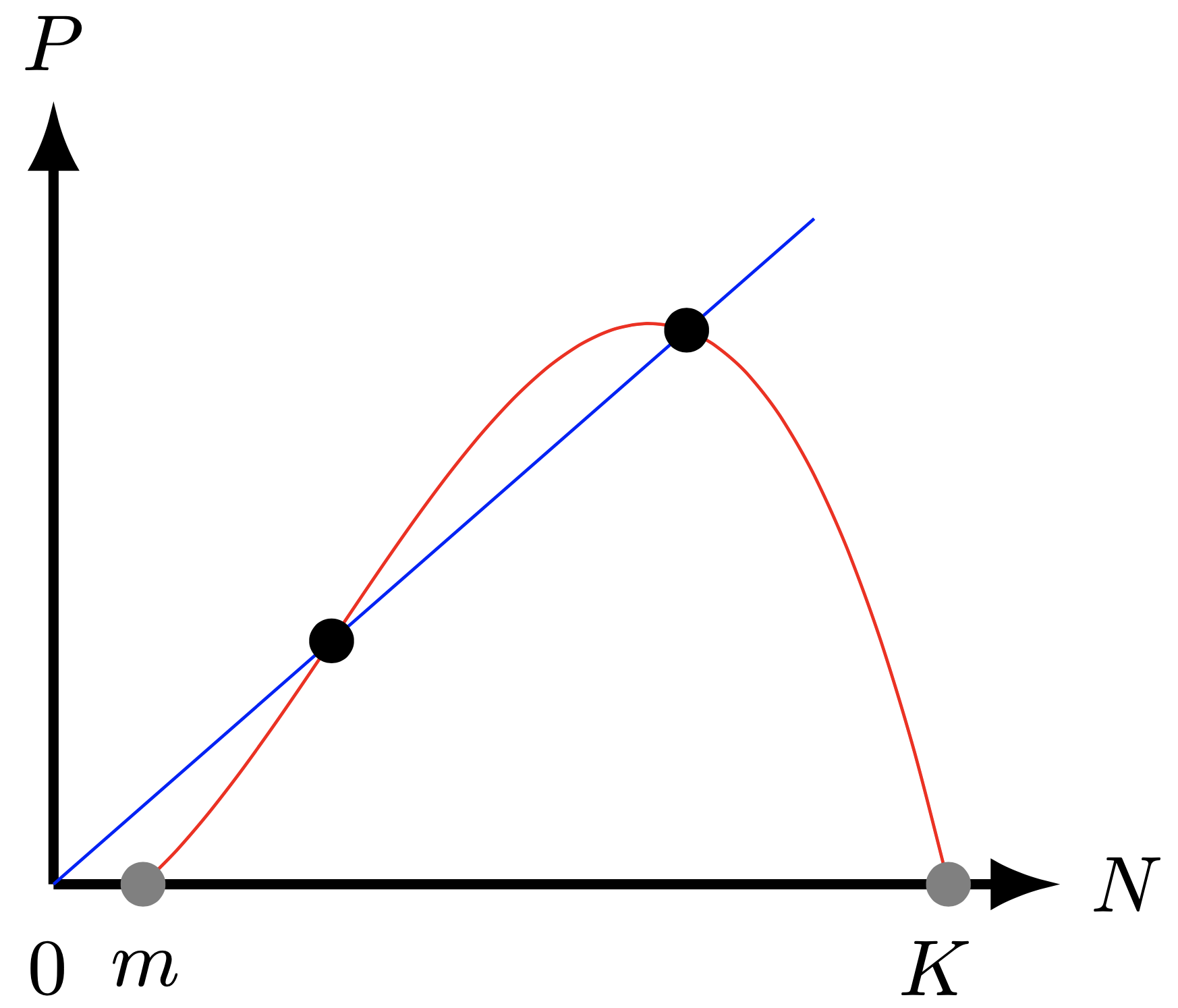} & \\
 \multirow{2}{*}{$m=0$} & \multirow{2}{*}{At most 2} & & & \\
 && \includegraphics[width=3.5cm]{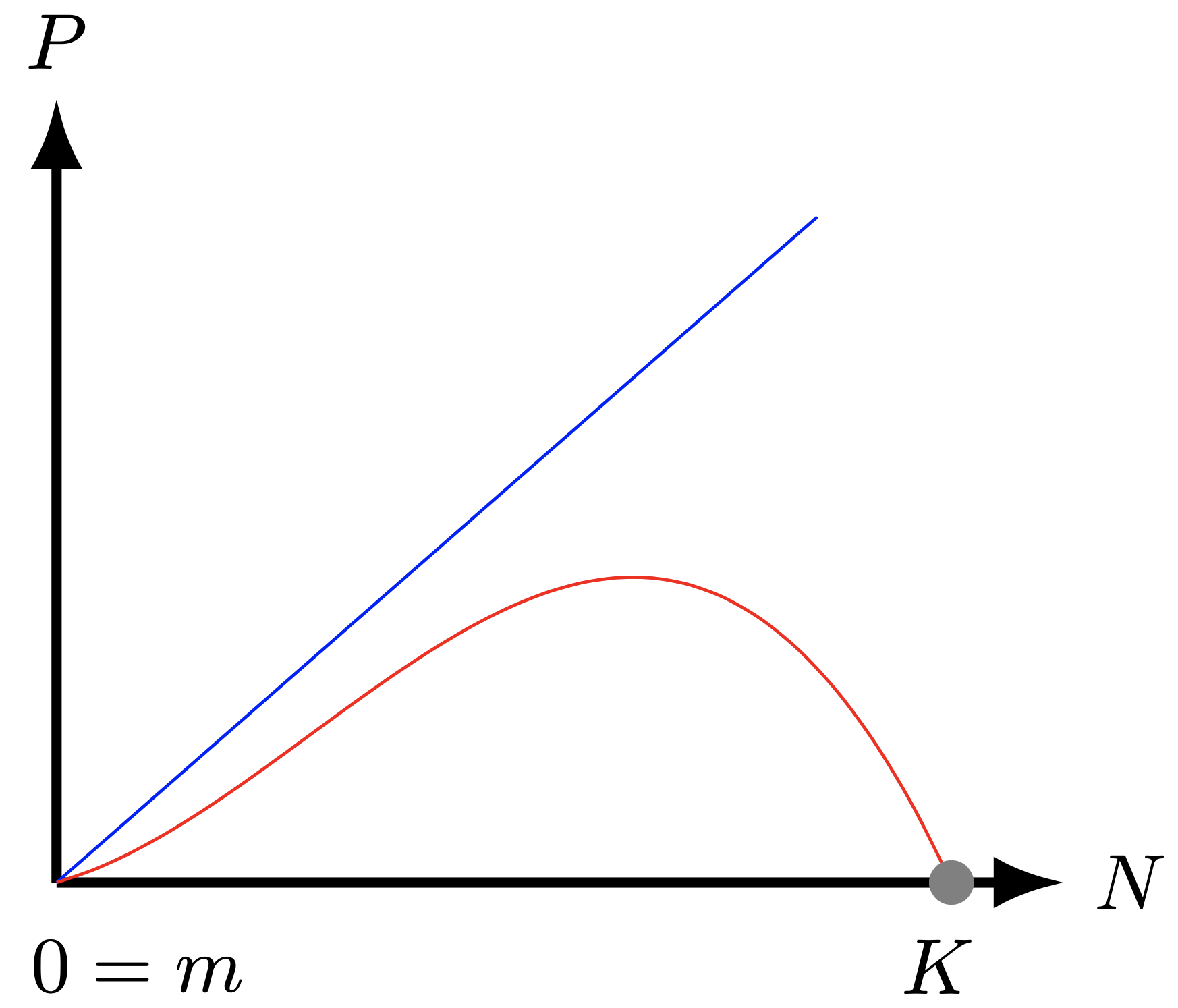} & \includegraphics[width=3.5cm]{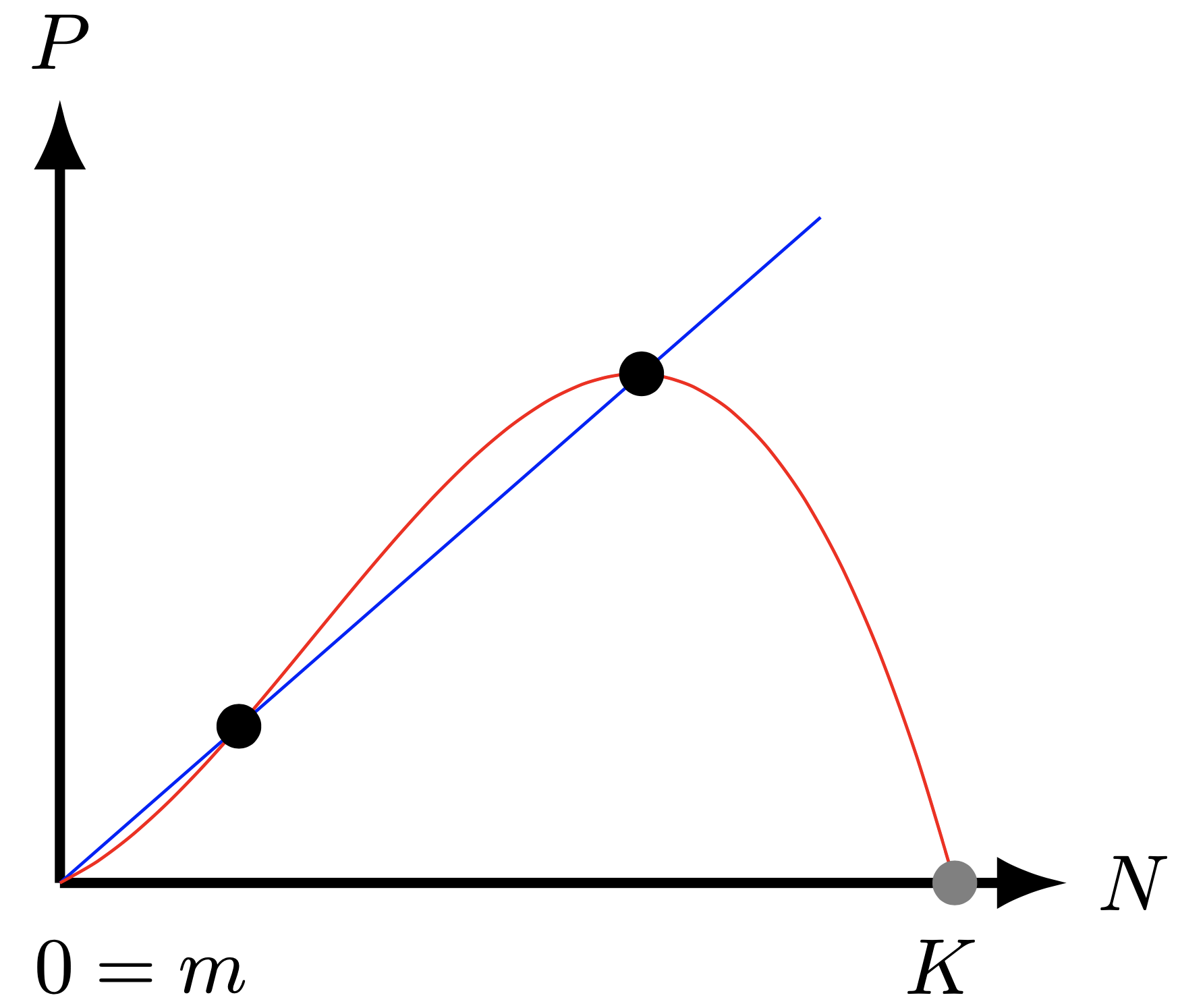} & \includegraphics[width=3.5cm]{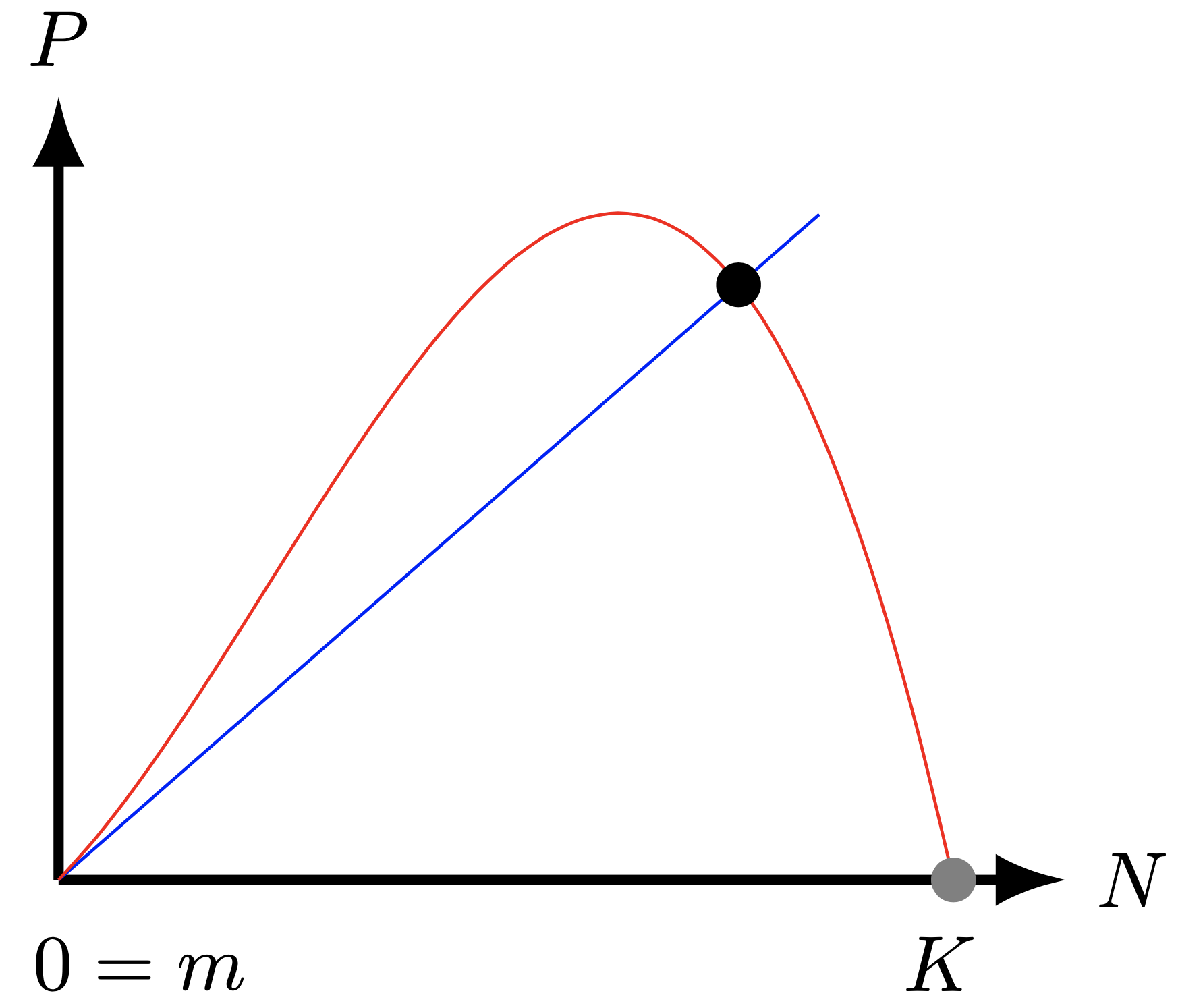}\\
 \multirow{2}{*}{$m<0$} & \multirow{2}{*}{At most 3} & & & \\
 && \includegraphics[width=3.5cm]{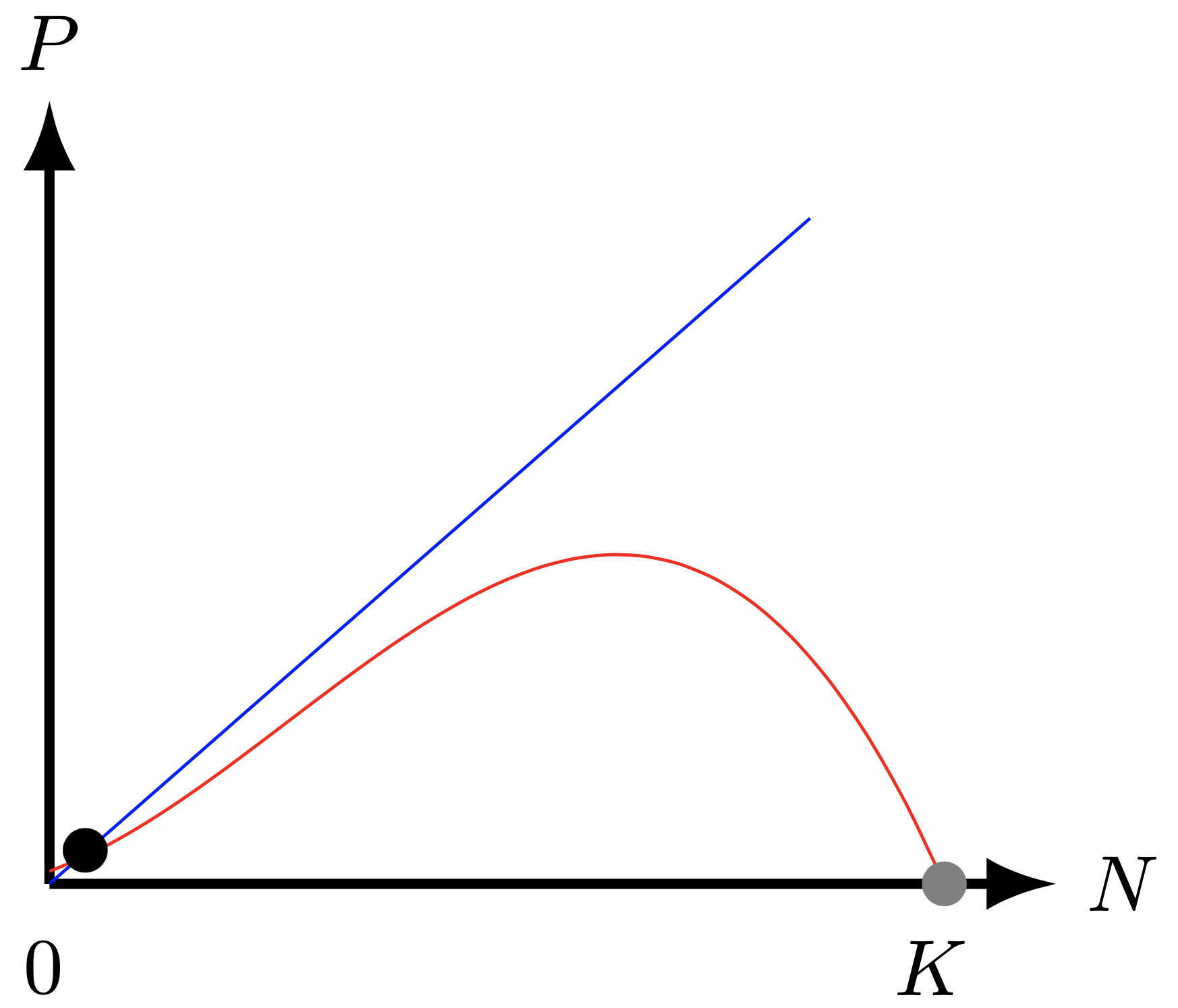} & \includegraphics[width=3.5cm]{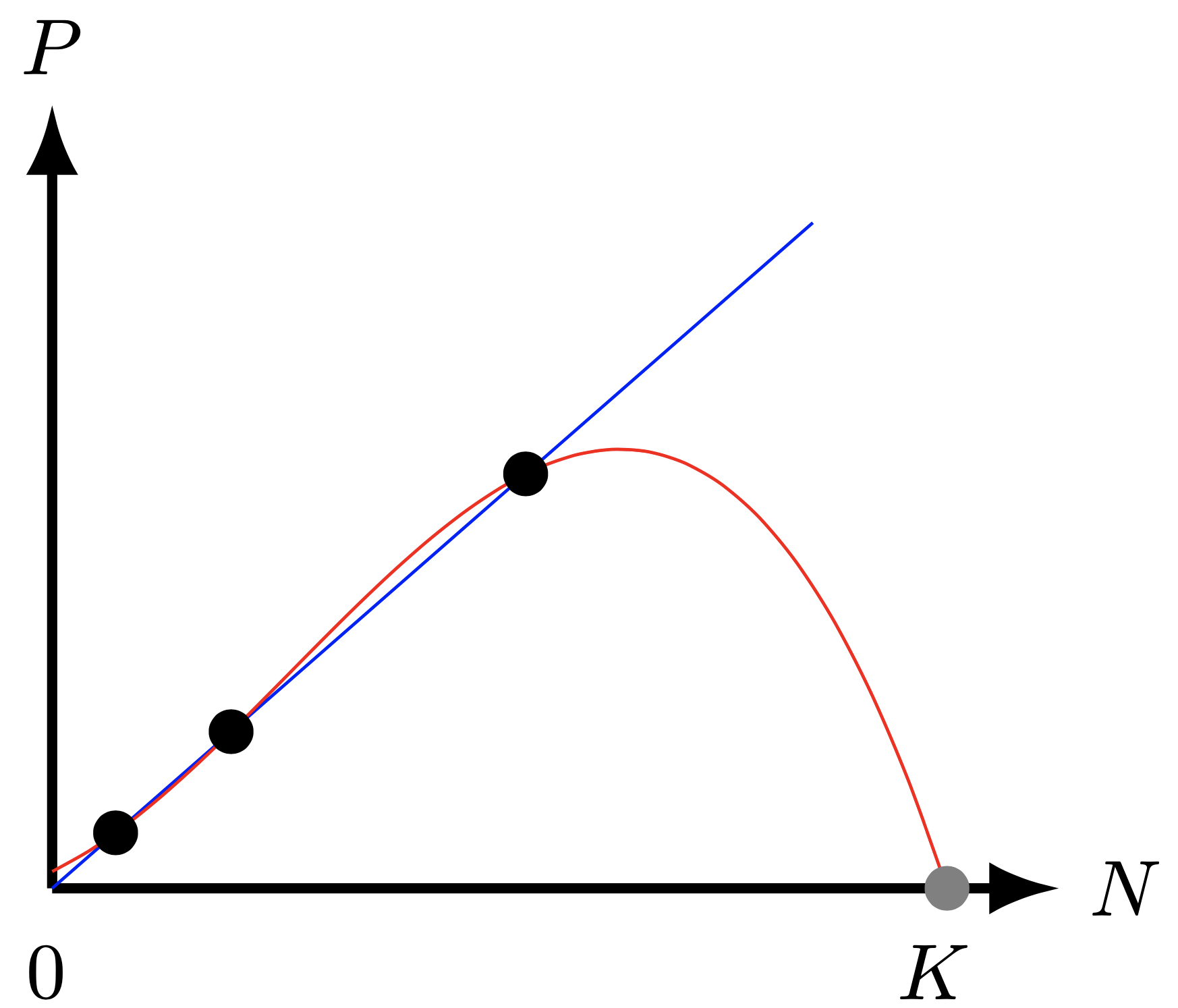} &  \includegraphics[width=3.5cm]{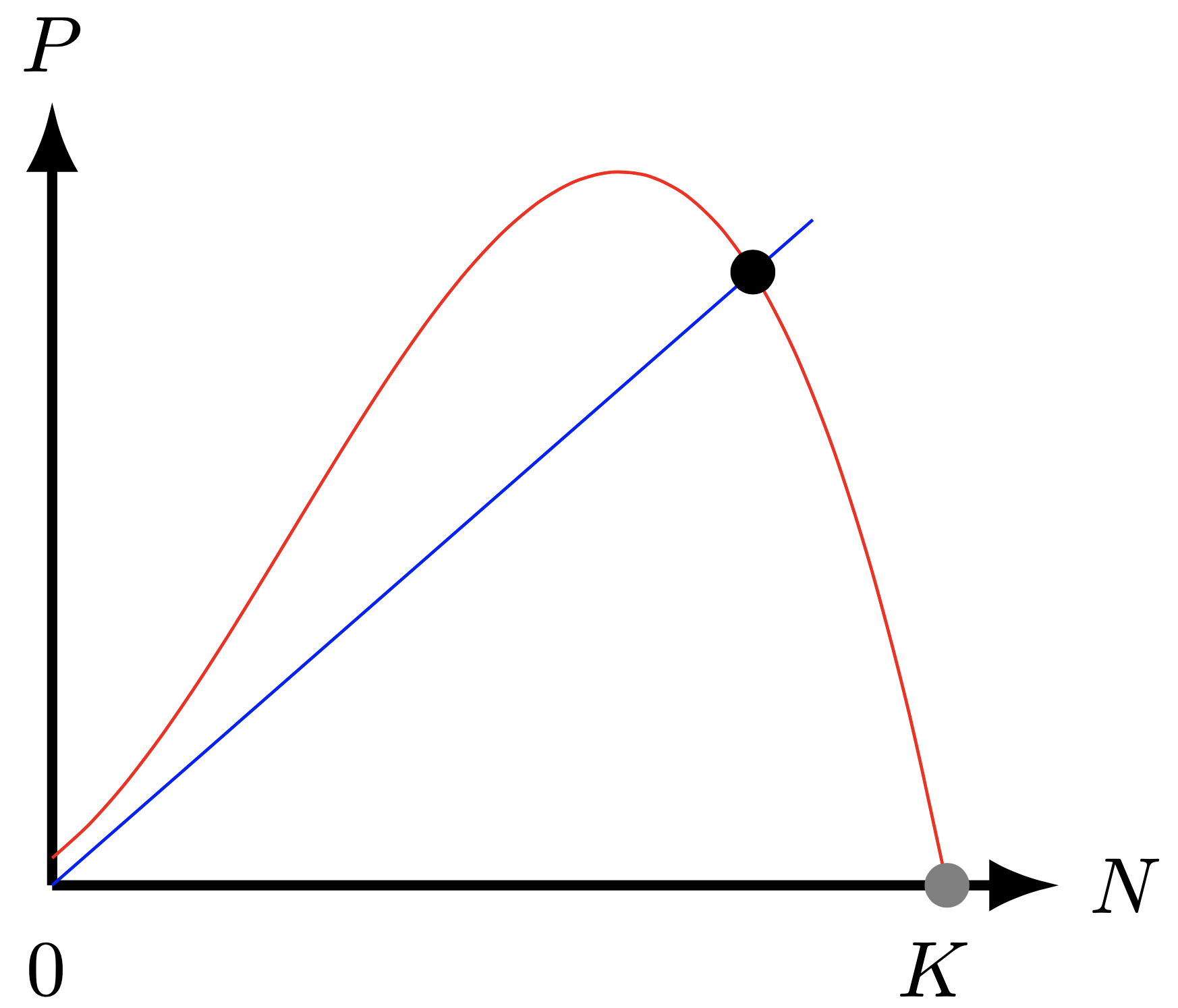}\\
\hline
\end{tabular}
\caption{Number of equilibria of system~\eqref{eq03} with strong ($m>0$), degenerate weak ($m=0$) and weak ($m<0$) Allee effect. The number of equilibria showed in the last row are obtained by varying the predation rate $q$.}\label{T1}
\end{table}

In Section~\ref{S01} we nondimensionalise the Leslie--Gower model with weak Allee effect and discuss the number of equilibria the model has in the first quadrant. The main mathematical difference between system~\eqref{eq02}, ~\eqref{eq03} with $m<0$, and ~\eqref{eq03} with $m\geq0$, is the fact that~\eqref{eq03} with $m<0$ has at most three positive equilibria\footnote{We will refer to equilibria with two positive entries as positive equilibria.} in the first quadrant instead of one for system~\eqref{eq02} and two for system~\eqref{eq03} with $m\geq0$~\cite{arancibia,gonzalez3,arancibia3}. These additional equilibrium gives rise to different type of bifurcations, such as saddle-node bifurcations, Bogdanov--Takens bifurcations, and homoclinic bifurcations. The main properties of the equilibria are studied in Section~\ref{S02}. In particular, we study the stability of the equilibria and present the conditions for which the model undergoes different type of bifurcations. Finally, in Section~\ref{con} we summarise the results and discuss the ecological implications of the model.

\section{The Model}\label{S01}
The Leslie--Gower model with weak Allee effect is given by~\eqref{eq03} with $m<0$, and we only consider the model in the domain $\Omega=\{\left(N,P\right)\in\mathbb{R}^2,N>0, P\geq0\}$ and $\left(r,K,q,a,s,h\right)\in \mathbb{R}{}{}_{+}^{6}$. The axes in system~\eqref{eq03} are invariant since~\eqref{eq03} is of Kolmogorov  type~\cite{dumortier} (i.e. $dN/dt=N\cdot W\left(N,P\right)$ and $dP/dt=P\cdot R\left(N,P\right)$). The equilibria of system~\eqref{eq03} are $\left(K,0\right)$ and $\left(x^*,y^*\right)$, which are the intersections of the nullclines \[P=hN \quad \text{and} \quad P=\dfrac{r}{q}\left( 1-\dfrac{N}{K}\right)\left(N+a\right)\left(N-m\right).\]

We follow the nondimensionalisation approach of~\cite{arancibia3} to simplify the analysis and remove the $N=0$ singularity for the model. We introduce $\bar{\Omega}=\{\left(u,v\right)\in\mathbb{R}^2,u\geq0, v\geq0\}$ and the dimensionless variables $\left(u,v,\tau\right)$ given by 
\begin{equation}\label{eq04}
\begin{aligned}
& \varphi :\bar{\Omega}\times\mathbb{R}\rightarrow \Omega\times\mathbb{R}~\text{where}~\varphi\left(u,v,\tau\right)=\left(\dfrac{N}{K},\dfrac{P}{hK},\dfrac{rKt}{u\left(u+\dfrac{a}{K}\right)}\right).
\end{aligned}  
\end{equation}
Observe that $\varphi$~\eqref{eq04} is a diffeomorphism which preserve the orientation of time~\cite{andronov,chicone}. Next, set $A:=a/K\in\left(0,1\right)$, $S:=s/\left(rK\right)$, $Q:=hq/\left(rK\right)$ and $M:=m/K$, then~\eqref{eq03} transforms into the nondimensionalised system
\begin{equation}\label{eq05}
\begin{aligned}
\dfrac{du}{d\tau} & =  \ u^2\left(\left(u+A\right)\left( 1-u\right)\left(u-M \right) -Qv\right)\,,\\
\dfrac{dv}{d\tau} & =  \ S\left(u+A\right)\left(u-v\right)v \,.
\end{aligned}  
\end{equation}

The $u$-nullclines of system~\eqref{eq05} are given by $u=0$ and $v=\left(u+A\right)\left(1-u\right)\left(u-M\right)/Q$, while the $v$-nullclines of interest are given by $v=0$ and $v=u$. Hence, the equilibria of system~\eqref{eq05} are $\left(0,0\right)$\footnote{Note that~\eqref{eq03} is singular along $N=0$, and the equilibria $(0,0)$ is thus also a singular point in~\eqref{eq03} with $m<0$. }, $\left(1,0\right)$, and the positive equilibria $\left(u^*,v^*\right)$ with $v^*=u^*$ and where $u^*$ is determined by the solution(s) of $\left(u+A\right)\left( 1-u\right)\left(u-M \right)/Q=u$, or, equivalently,
\begin{equation}\label{eq06}
g(u):=u^3-T\left(A,M\right)u^2-L\left(A,M,Q\right)u+AM =0\,,
\end{equation}
with $T\left(A,M\right)=1-A+M$ and $L\left(A,M,Q\right)=A\left(M+1\right)-Q-M$. 
Since $g(0)=AM<0$ and $g(1)=Q>0$, there is always at least one positive equilibrium $(u^*,u^*)$ with $0<u^*<1$.
To obtain information about potential other positive equilibria, we divide~\eqref{eq06} by $\left(u-u^*\right)$ and obtain the quadratic equation 
\begin{equation}\label{eq07}
u^2+u\left(u^*-T\left(A,M\right)\right)+u^*\left(u^*-T\left(A,M\right)\right)-L\left(A,M,Q\right)=0.
\end{equation}
Note that as system~\eqref{eq05} is topologically equivalent to system~\eqref{eq03}, the solution $u^*$ of~\eqref{eq06} is related to a population equilibrium of system~\eqref{eq03}. Therefore, a positive equilibrium point $P=(u^*,u^*)$ in system~\eqref{eq05} correspond to a coexisting equilibrium point in system~\eqref{eq03}.
\begin{figure}
\centering	
\begin{subfigure}[b]{0.3\textwidth}\centering
\caption{$Q<Q^-$.}\label{F02_a}
\includegraphics[width=5cm]{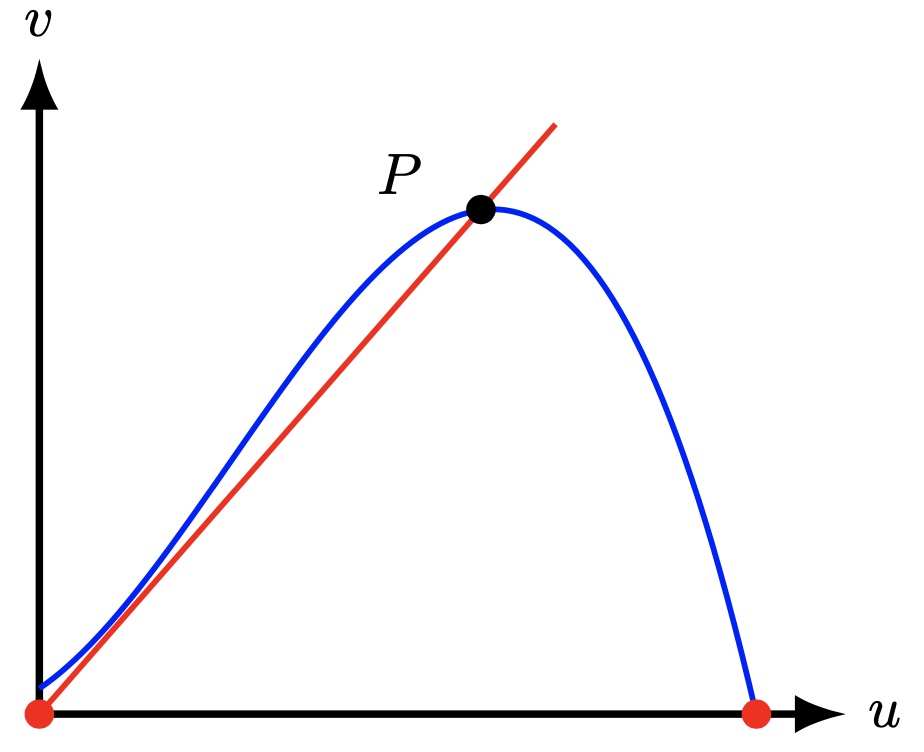}
\end{subfigure}
\hfill
\begin{subfigure}[b]{0.3\textwidth}\centering
\caption{$Q=Q^-$.}\label{F02_b}
\includegraphics[width=5cm]{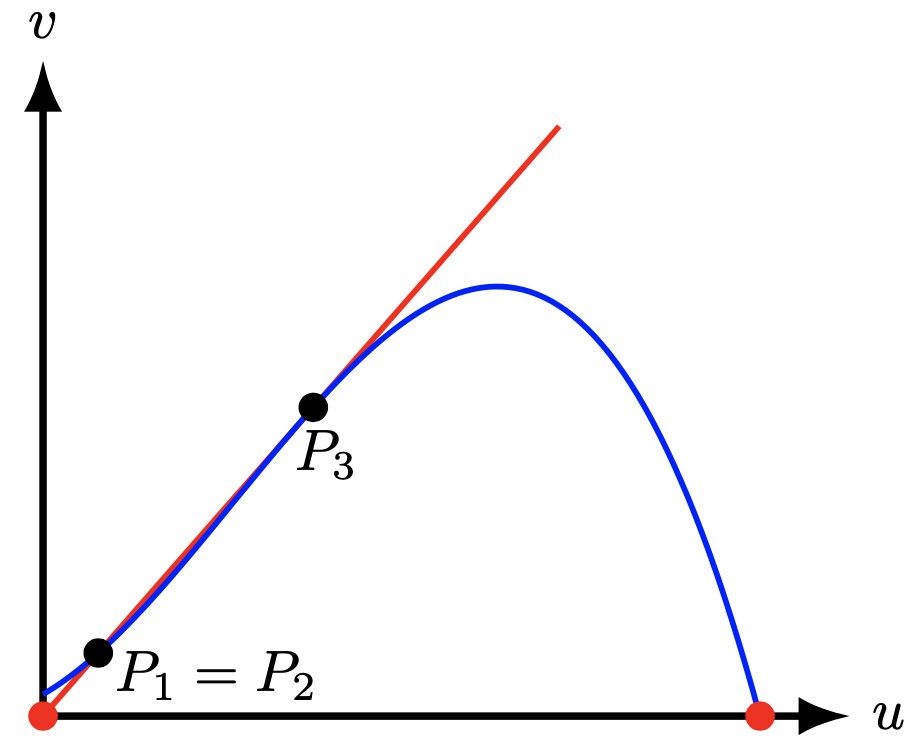}
\end{subfigure}
\hfill
\begin{subfigure}[b]{0.3\textwidth}\centering
\caption{$Q^-<Q<Q^+$.}\label{F02_c}
\includegraphics[width=5cm]{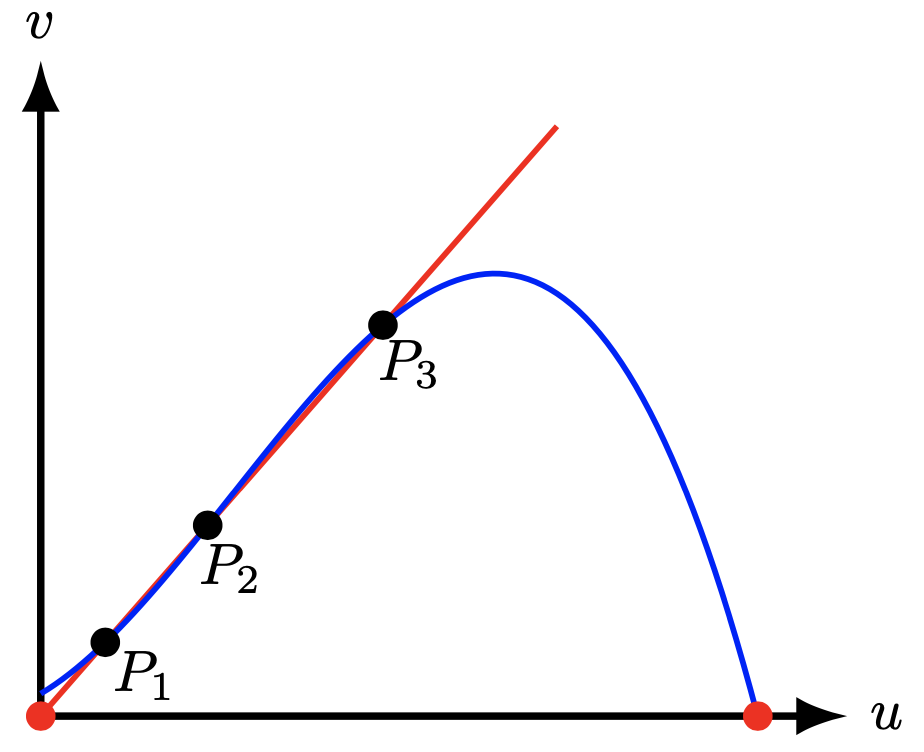}
\end{subfigure}
\hfill
\\
\vspace{0.2in}
\begin{subfigure}[t]{.496\linewidth}\centering
\resizebox{\linewidth}{!}{	
\includegraphics[width=5cm]{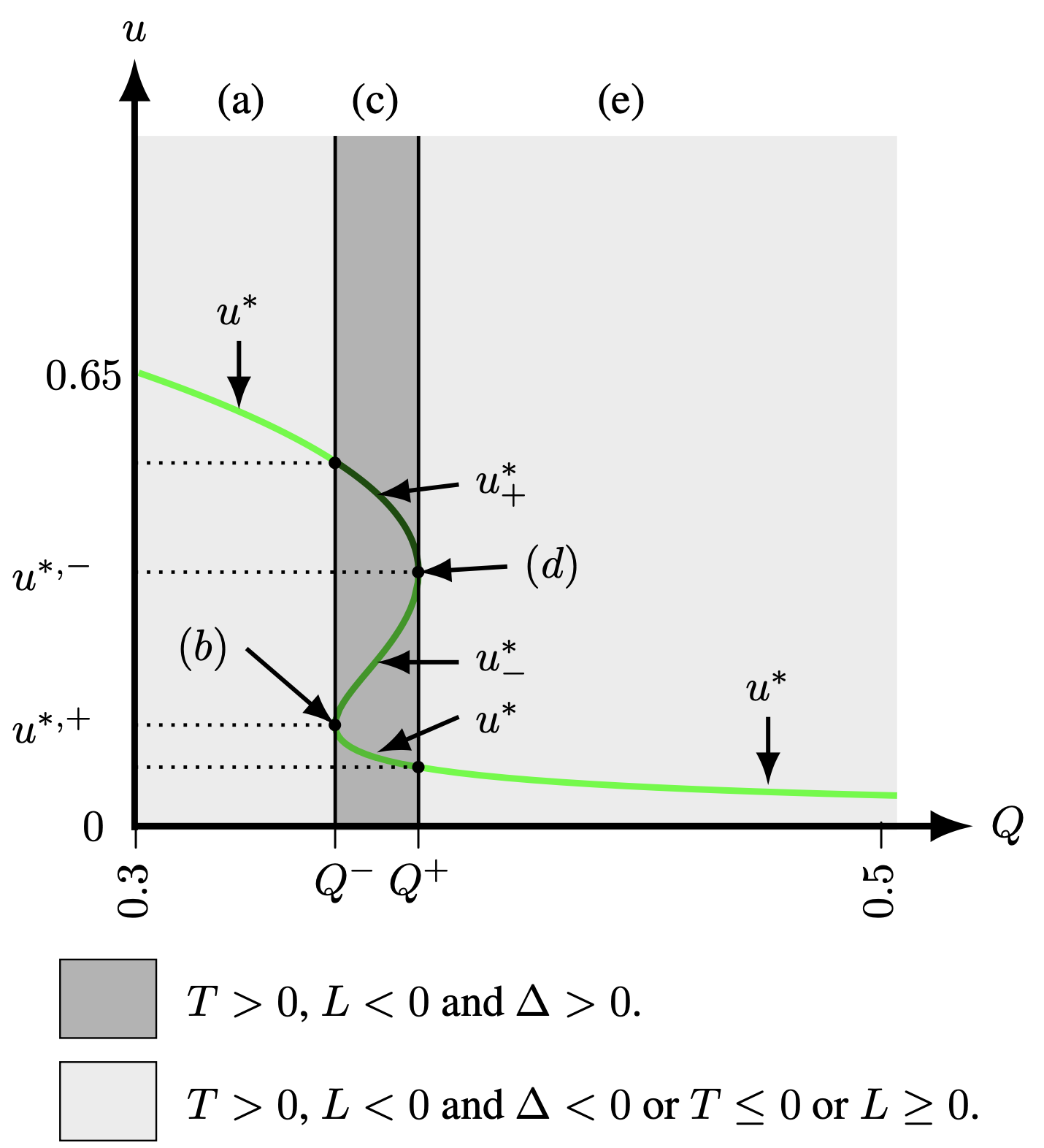}}
\end{subfigure}
\hfill
\\
\vspace{0.2in}
\begin{subfigure}[b]{0.3\textwidth}\centering
\caption{$Q=Q^+$.}\label{F02_d}
\includegraphics[width=5cm]{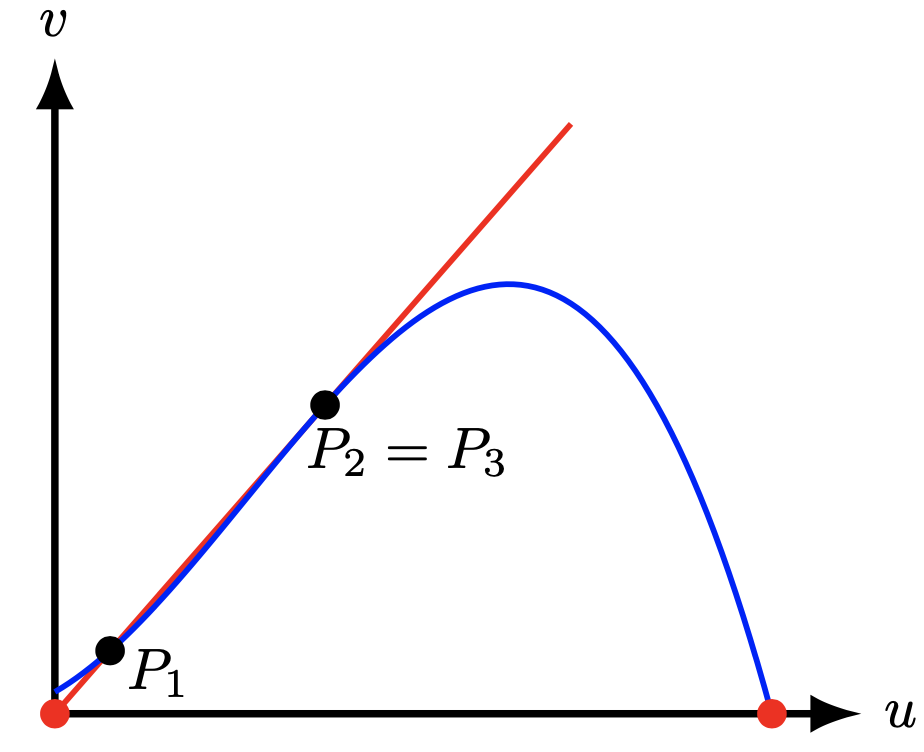}
\end{subfigure}
\begin{subfigure}[b]{0.3\textwidth}\centering
\caption{$Q>Q^+$.}\label{F02_e}
\includegraphics[width=5cm]{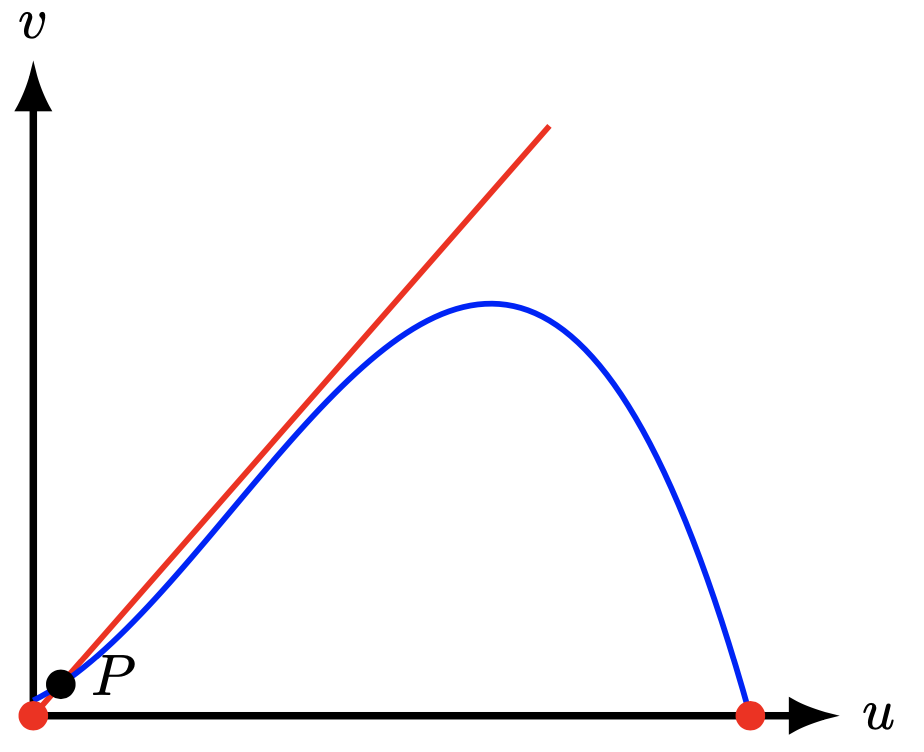}
\end{subfigure}
\caption{The schematic diagram of the number of positive roots of equation~\eqref{eq06} in the $(u,Q)$-space for $(A,M)=(1/10,-1/10)$ fixed and with $T=1-A+M$, $L=A\left(M+1\right)-Q-M$, $\Delta=\left(u^*-T\left(A,M\right)\right)^2-4\left(u^*\left(u^*-T\left(A,M\right)\right)-L\left(A,M,Q\right)\right)$ and $u^*$ is determined by the solution(s) of~\eqref{eq06}. At $Q=Q^\pm := (73 \pm \sqrt5)/200$, that is, at $Q^-\approx0.35381966$ and $Q^+\approx0.376180$, 
two positive equilibria collapse. For $Q<Q^-$ or $Q>Q^+$, equation~\eqref{eq06} has only one positive root (light grey regions). Hence, system~\eqref{eq05} has only one positive equilibrium $P=(u^*,u^*)$. For $ Q^-<Q<Q^+$, equation~\eqref{eq06} has three positive roots (dark grey region) and system~\eqref{eq05} thus has three positive equilibria.}\label{F02}
\end{figure}

\begin{lemm}\label{NEP}
The positive roots of $g(u)$~\eqref{eq06} lie in $(0,1)$, and 
\begin{enumerate}[label=(\Roman*)]
\item \label{NEP1} 
if $T\left(A,M\right)\leq0$ or $L\left(A,M,Q\right)\geq0$,
then~$g(u)$~\eqref{eq06} has one positive root and system~\eqref{eq05} thus has one positive equilibrium, 
see~\ref{F02_a} and~\ref{F02_e} in Figure~\ref{F02}.
\item \label{NEP2} if $T\left(A,M\right)>0$, $L\left(A,M,Q\right)<0$ and 
\begin{enumerate}[label=(\roman*)]
\item \label{NEP2a} $\Delta<0$, where $\Delta$ is the discriminant of~\eqref{eq07} (with respect of the dependent variable $u$) given by
\begin{equation}\label{delta}
\Delta:=\left(u^*-T\left(A,M\right)\right)^2-4\left(u^*\left(u^*-T\left(A,M\right)\right)-L\left(A,M,Q\right)\right),
\end{equation} 
then~$g(u)$~\eqref{eq06} has one positive root and system~\eqref{eq05} thus has one positive equilibrium, see~\ref{F02_a} and~\ref{F02_e} in Figure~\ref{F02};
\item \label{NEP2b} $\Delta\geq0$~\eqref{delta}, then~$g(u)$~\eqref{eq06} has three roots (counting multiplicity) and system~\eqref{eq05} thus has three positive equilibria, see~\ref{F02_b} -- \ref{F02_d} in Figure~\ref{F02}. 
\end{enumerate}
\end{enumerate}
\end{lemm}

\begin{proof}
The cubic equation $g(u)$~\eqref{eq06} always has at least one root in $(0,1)$, since $g(0)=AM<0$ and $g(1)=Q>0$, and 
the first case of the lemma, case \ref{NEP1}, thus immediately follows from Descartes’ rule of signs~\cite[e.g]{curtiss}. 

Case \ref{NEP2}\ref{NEP2a} also follows directly from Descartes’ rule of signs, which for $T\left(A,M\right)>0$ and $L\left(A,M,Q\right)<0$ states that $g(u)$~\eqref{eq06} has one or three positive roots (counting multiplicity), and the observation that the quadratic equation~\ref{eq06} has two complex roots as its discriminant is negative.

In the last case, case \ref{NEP2}\ref{NEP2b}, $g(u)$~\eqref{eq06} has three real roots: $u^*$ and
\begin{equation}\label{eq08}
\begin{aligned}
u^*_-&=\dfrac{1}{2}\left(T\left(A,M\right)-u^*-\sqrt{\Delta}\right),\quad u^*_+=\dfrac{1}{2}\left(T\left(A,M\right)-u^*+\sqrt{\Delta}\right).
\end{aligned} 
\end{equation} 
The latter two being the roots of~\eqref{eq07}. What remains to show is that $u^*_\pm \in (0,1)$. To show this, we look at the derivative of $g(u)$
$$
g'(u) = 3u^2-2T\left(A,M\right)u-L\left(A,M,Q\right).
$$
The assumptions $T\left(A,M\right)>0$ and $L\left(A,M,Q\right)<0$ imply that $g'(u) \ge 0$ for $u \leq 0$. Furthermore, $T\left(A,M\right) \in (0,1)$ and, for $u \geq T$, we thus also get
$$
g'(u) = 3u^2-2T\left(A,M\right)u-L\left(A,M,Q\right) > 2u(u-T\left(A,M\right))+u^2> 2u(u-T\left(A,M\right)) \geq 0\,.
$$ 
In other words, under the assumptions $T\left(A,M\right)>0$ and $L\left(A,M,Q\right)<0$, the cubic function $g(u)$ is strictly increasing outside the interval $(0,T\left(A,M\right))$ (with $T\left(A,M\right)<1$). Since $g(0)=AM<0$ and $g(1)=Q>0$, it thus follows that $u^*_\pm \in (0,1)$.  
\end{proof}

We observe that none of these equilibria explicitly depend on the system parameter $S$. Moreover, only $L\left(A,M,Q\right)$ in $g(u)$~\eqref{eq06} depends directly on $Q$. Therefore, $S$ and $Q$ are natural candidates to act as bifurcation parameters. 
For instance, when 
two of the three roots of $g(u)$ coincide, say $u_-^*$ and $u_+^*$, we change from having three roots to one root (upon changing $Q$). So, at this point we have
$$
u_-^* = u_+^* \implies u^*= u^{*,\pm}:=\dfrac{1}{3}\left(T\left(A,M\right)\pm2\sqrt{T\left(A,M\right)^2+3L\left(A,M,Q\right)}\right). 
$$
Implementing this into $g(u^{*,\pm})=0$ results in an analytic implicit expression $Q=Q^{\pm}(A,M)$ at which the number of roots changes. For example, for $(A,M)=(1/10,-1/10)$ as in Figure~\ref{F02} we get $Q^{\pm}(1/10,-1/10)=  (73 \pm \sqrt{5})/200$.\footnote{The same expressions for $Q^{\pm}$ are obtained when we compute two other roots coincide.}. 

In case \ref{NEP2}\ref{NEP2b} of Lemma~\ref{NEP} we find that $g(u)$ has three roots in $(0,1)$: $u^*$ and $u^*_\pm$ \eqref{eq08}, with $u^*_-\leq u^*_+$. The latter two roots depend on the first root $u^*$, but, {\em a priori}, we do not know anything about parity of $u^*-u^*_\pm$. However, the roots of $g(u)$ do not depend on which of the three roots we pick {\emph{a priori}} as $u^*$ in the lemma. Hence, we can, without loss of generality, assume in the remainder of this manuscript that $0<u^*\leq u^*_-\leq u^*_+<1$. However, see the proof of Corollary~\ref{C02} where we utilise this {\emph{independence}} of our initial pick.     

\section{Main Results}\label{S02}
In this section, we discuss the main results related to system~\eqref{eq05}. That is, we  discuss the nature of the equilibria (Section~\ref{S:NEP}) and their bifurcations (Section~\ref{BA}). First we observe that from~\cite[Theorem~2]{arancibia3} it instantly follows that all solutions of~\eqref{eq05} which are initiated in the first quadrant are bounded and end up in 
\begin{equation}\label{phi}
\Phi=\{\left(u,v\right),\ 0<u\leq1,\ 0\leq v\leq1\}.
\end{equation}

\subsection{The Nature of the equilibria}\label{S:NEP}
To determine the stability of the equilibria on the axes of interest $(0,0)$ and $(1,0)$ we compute the Jacobian matrix of system~\eqref{eq05}
\begin{equation}\label{eq09}
\begin{aligned}
J\left(u,v\right)=\begin{pmatrix}
-uJ_{11}  & -Qu^2 \\ 
Sv\left( A+2u-v\right)  &  S\left(u-2v\right)\left( A+u\right) 
\end{pmatrix},
\end{aligned}
\end{equation}
with $J_{11}=4Au^2-4Mu^2+2AM-3Au+3Mu+2Qv-4u^2+5u^3-3AMu$. The determinant and the trace of the Jacobian matrix~\eqref{eq09} are:
\begin{equation}\label{det_ax}
\begin{aligned}
\det\left(J\left(u,v\right)\right)=&-J_{11}Su\left(u-2v\right)\left(A+u\right)+QSu^2v\left(A+2u-v\right),\\
\tr\left(J\left(u,v\right)\right)=&-uJ_{11} +S\left(u-2v\right)\left( A+u\right)
\end{aligned}
\end{equation}

\begin{lemm}\label{TEST}
The equilibrium $\left(0,0\right)$ is a non-hyperbolic saddle point and $\left(1,0\right)$ is a saddle point.
\end{lemm}
\begin{proof}
Since $\det\left(J\left(1,0\right)\right)=-S\left(1-M\right)\left(A+1\right)^2<0$ it immediately follows that $(1,0)$ is a saddle point. To prove the results for the origin we follow the methodology used to desingularise the origin showed in~\cite[Lemma~2]{arancibia3}. First, we observe that setting $u=0$ in system~\eqref{eq05} the second equation reduces to $dv/dt=-v^2\left(SA\right)<0$ for $v\neq0$. That is, any trajectory starting along the $v$-axis converges to the origin $\left(0,0\right)$. Also, the Jacobian matrix, $J\left(0,0\right)$, is the zero matrix. Hence, the origin $\left(0,0\right)$ is a non-hyperbolic equilibrium of system~\eqref{eq05}. Since the horizontal blow-up in~\eqref{eq05} does not give any further information we omit the details and only consider the vertical method to desingularise the origin and study the dynamics of this equilibrium. The vertical \emph{blow-up} given by the transformation and the time rescaling
\begin{equation}
	\left(u,v\right)\to \left(xy,y\right)~\text{and}~\tau \to \dfrac{t}{y},\label{eq10}
	\end{equation}
respectively. This transformation is well-defined for all values of $u$ and $v$ except for $v=0$ and `blows-up'  the origin of system~\eqref{eq05} into the entire $x$-axis. Our goal is to analyse the equilibrium on the positive half axis $x\geq0$, $y=0$ of the transformed system, which is given by:
	\begin{equation}\label{eq12}
	\begin{aligned}
	\dfrac{dx}{dt}&=x\left(S\left(1-x\right)\left(A+xy\right)+x\left(M-xy\right)\left(xy-1\right)\left(A+xy\right)-Qxy\right),\\
	\dfrac{dy}{dt}&=Sy\left(x-1\right)\left(xy+A\right).
	\end{aligned}
	\end{equation}
System~\eqref{eq12} has up to two equilibria on the non negative $x$-axis of the form $\left(x,0\right)$ with $x\geq 0$. The origin $O_{xy}=\left(0,0\right)$ and a second equilibrium $I_x=\left(\mu,0\right)$ with $\mu=S/\left(S+M\right)$ if $S>|M|$. Their corresponding Jacobian matrix $J_*$ evaluated at $O_{xy}$ and $I_x$ are:
	\begin{equation*}
	J_*\left(O_{xy}\right)=\begin{pmatrix} AS & 0 \\ 0 & -AS \end{pmatrix}
	\end{equation*}
with eigenvalues $\lambda_1\left(O_{xy}\right)=AS$ and $\lambda_2\left(O_{xy}\right)=-AS$ and
	\begin{equation*}
	J_*\left(I_x\right)=\begin{pmatrix} -AS & \dfrac{S^2\left(AS\left(1+M\right)-Q\left(M+S\right)\right)}{\left(M+S\right)^3}\\ & \\ 0 &- \dfrac{AMS}{M+S}\end{pmatrix}
	\end{equation*}
with eigenvalues $\lambda_1\left(I_x\right)=-AS$ and $\lambda_2\left(I_x\right)=-AMS/\left(M+S\right)>0$ since $S>|M|$. It follows that both $O_{xy}$ and $I_x$ are saddle points in system~\eqref{eq12}. Moreover, a branch of the unstable  manifold $W^u\left(I_x\right)$  of the equilibrium $I_x$ is in the half-plane $y>0$, as illustrated in the left panel of Figure~\ref{F03}. Furthermore, the other local invariant curves are the axes $x=0$ and $y=0$. Hence, taking the  inverse of~\eqref{eq10}, the line $y=0$, including the point $I_x$, collapses to the origin $O_{uv}$ of~\eqref{eq05}, the line $x=0$ is mapped to $u=0$ and, $W^u\left(I_x\right)$ is locally mapped to the curve $\Gamma^u$, the unstable manifold of $O_{uv}$.  Both curves are locally represented by the eigenvectors associated to the positive eigenvalues tangent to the curves at $I_x$ and $O_{uv}$, respectively. Since the orientation of the orbits in the first quadrant is preserved by ~\eqref{eq10}, it follows that the origin $O=\left(0,0\right)$ is a local saddle of~\eqref{eq05}. The qualitative dynamics in a neighbourhood of the origin $O_{uv}$ in~\eqref{eq05} is illustrated in the right panel of Figure~\ref{F03}. If $S=|M|$, then $I_x$ collapses to the origin $O_{xy}$ and if $S<|M|$, then the equilibrium $I_x$ is in the half-plane $y<0$ which is outside of $\Phi$~\eqref{phi}. So,~\eqref{eq12} has one non negative equilibrium $(0,0)$ which still a saddle. 
\end{proof}

\begin{figure} 
\centering
\begin{subfigure}[b]{0.35\textwidth}
\centering
\includegraphics[width=6.2cm]{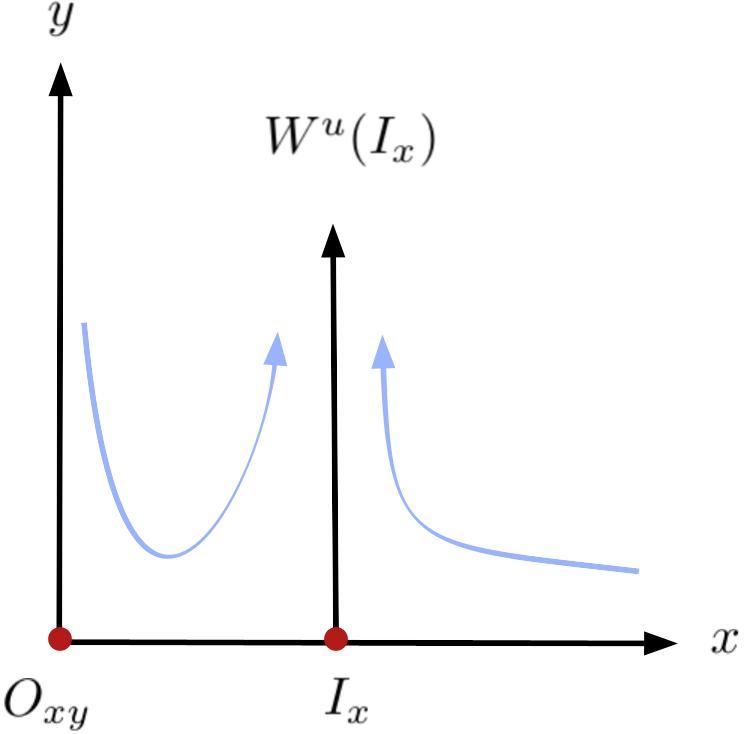}
\caption{Blow-up}\label{F03_a}
\end{subfigure}
\begin{subfigure}[b]{0.35\textwidth}\centering
\includegraphics[width=6cm]{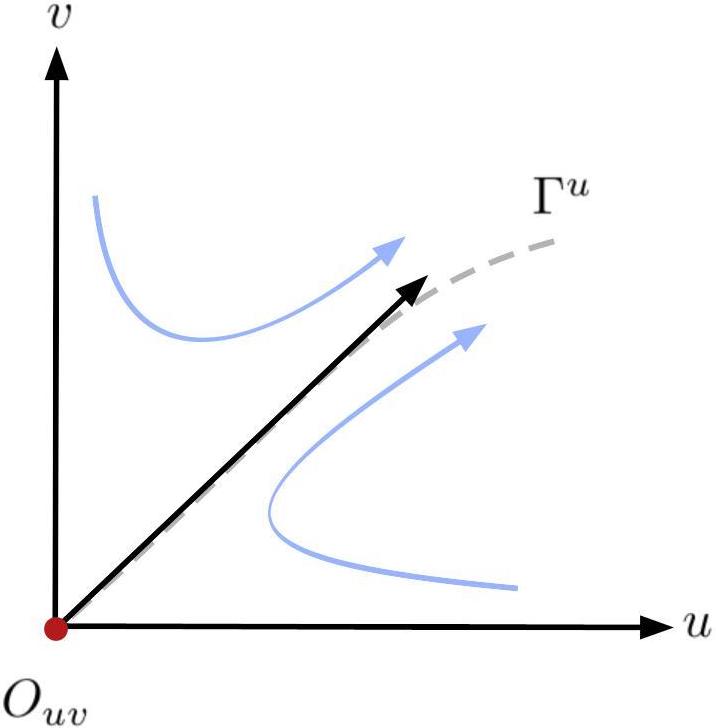}
\caption{Blow-down}\label{F03_b}
\end{subfigure}
\caption{Diagram of the horizontal blow-up and blow-down in a neighbourhood of the origin $(0,0)$.}\label{F03}
\end{figure}

Next, we consider the stability of the positive equilibria of system~\eqref{eq05}. Note that these equilibria are the intersection of the nullcline $u=v$ such that $\left(u+A\right)\left(1-u\right)\left(u-M\right)=Qu$. Therefore, the Jacobian matrix of system~\eqref{eq05} becomes
\begin{equation}\label{eq13}
\begin{aligned}
J\left(u,u\right)=\begin{pmatrix}
u^2\left(\left(1-u\right)\left(u-M\right)-\left(u+A\right)\left(u-M\right)+\left(1-u\right)\left(u+A\right)\right) & -Qu^2 \\ 
Su\left( A+u\right)  &  -Su\left( A+u\right) 
\end{pmatrix}.
\end{aligned}
\end{equation}
The determinant and the trace of the Jacobian matrix~\eqref{eq13} are:\\
\begin{equation}\label{eq14}
\begin{aligned}
\det\left(J\left(u,u\right)\right) &=Su^2\left(A+u\right)\left(u^2\left(2u-T\left(A,M\right)\right)-AM\right),\\
\tr\left(J\left(u,u\right)\right) &=u\left(\left(\left(1-u\right)\left(u-M\right)+\left(u+A\right)\left(1-2u+M\right)\right)u-S\left(A+u\right)\right).
\end{aligned}
\end{equation}
So, the parity of the determinant~\eqref{eq14} depends on the sign of $u^2\left(2u-T\left(A,M\right)\right)-AM $ and the parity of the trace~\eqref{eq14} depends on the sign of
\begin{equation}\label{s_tr}
\left(\left(1-u\right)\left(u-M\right)+\left(u+A\right)\left(1-2u+M\right)\right)u-S\left(A+u\right).
\end{equation}
This instantly yields the following result:
\begin{lemm}\label{L01new}
A positive equilibrium $P=(\tilde{u},\tilde{u})$ of~\eqref{eq05} will be
\begin{enumerate}[label=(\roman*)]
\item \label{L01newa} a saddle point if $(\tilde{u})^2\left(2\tilde{u}-T\left(A,M\right)\right)-AM < 0$;
\item \label{L01newb} a repeller if $(\tilde{u})^2\left(2\tilde{u}-T\left(A,M\right)\right)-AM > 0$ and $S<\dfrac{\tilde{u}\left(\left(1-\tilde{u}\right)\left(\tilde{u}-M\right)+\left(\tilde{u}+A\right)\left(1-2\tilde{u}+M\right)\right)}{\left(A+\tilde{u}\right)}=:f(\tilde{u})$; and
\item \label{L01newc} an attractor if $(\tilde{u})^2\left(2\tilde{u}-T\left(A,M\right)\right)-AM > 0$ and $S>\dfrac{\tilde{u}\left(\left(1-\tilde{u}\right)\left(\tilde{u}-M\right)+\left(\tilde{u}+A\right)\left(1-2\tilde{u}+M\right)\right)}{\left(A+\tilde{u}\right)}$.\\
\end{enumerate}
\end{lemm}
\begin{corr}\label{Cnew}
If $T(A,M)^3<-27AM$, then a positive equilibrium $P=(\tilde{u},\tilde{u})$ of~\eqref{eq05} is not a saddle. If for a positive equilibrium $P=(\tilde{u},\tilde{u})$ we have that $\tilde{u}> T(A,M)+\sqrt{T(A,M)^2+3(A-M+AM)}$, then this equilibrium is not a repeller.
\end{corr}
\begin{proof}
The first statement follows directly from the observation that $(\tilde{u})^2\left(2\tilde{u}-T\left(A,M\right)\right)-AM$ is minimal for nonnegative $\tilde{u}$ at $\tilde{u}=\hat{u}:=\max\{0,T(A,M)/3\}$. At this point $(\tilde{u})^2\left(2\tilde{u}-T\left(A,M\right)\right)-AM$ simplifies to $\min\{-AM, -T(A,M)^3/27-AM\}$. Hence,  $(\tilde{u})^2\left(2\tilde{u}-T\left(A,M\right)\right)-AM>0$ for all $\tilde{u}$ if $T(A,M)^3<-27AM$.
The second statement follows directly from the observation that for $\tilde{u} > T(A,M)+\sqrt{T(A,M)^2+3(A-M+AM)}$ we have that $f(\tilde{u})<0$. 
\end{proof}

In addition, $f(\tilde{u})$ has a maximum for positive $\tilde{u}$. Hence, a value for $S$ larger than this maximum again yields that the associated equilibrium cannot be a repeller.
For given $A$ and $M$ this maximum value can be explicitly computed. For instance, this maximum value is $361/1200$ for $A=1/10$ and $M=-1/10$.   

In the case that there is only one positive equilibrium, Lemma~\ref{L01new} simplifies to the following.
\begin{corr}\label{C01}
Let the system parameters of~\eqref{eq05} be such that the conditions of case~\ref{NEP1} or case~\ref{NEP2}\ref{NEP2a} of Lemma~\ref{NEP} are met. Then, system~\eqref{eq05} has only one positive equilibrium $P=(u^*,u^*)$ which is a repeller or an attractor.
If the positive equilibrium is a repeller, then it is surrounded by a stable limit cycle.
\end{corr}
\begin{figure} 
\centering
\begin{subfigure}[b]{0.48\textwidth}\centering
\includegraphics[width=8cm]{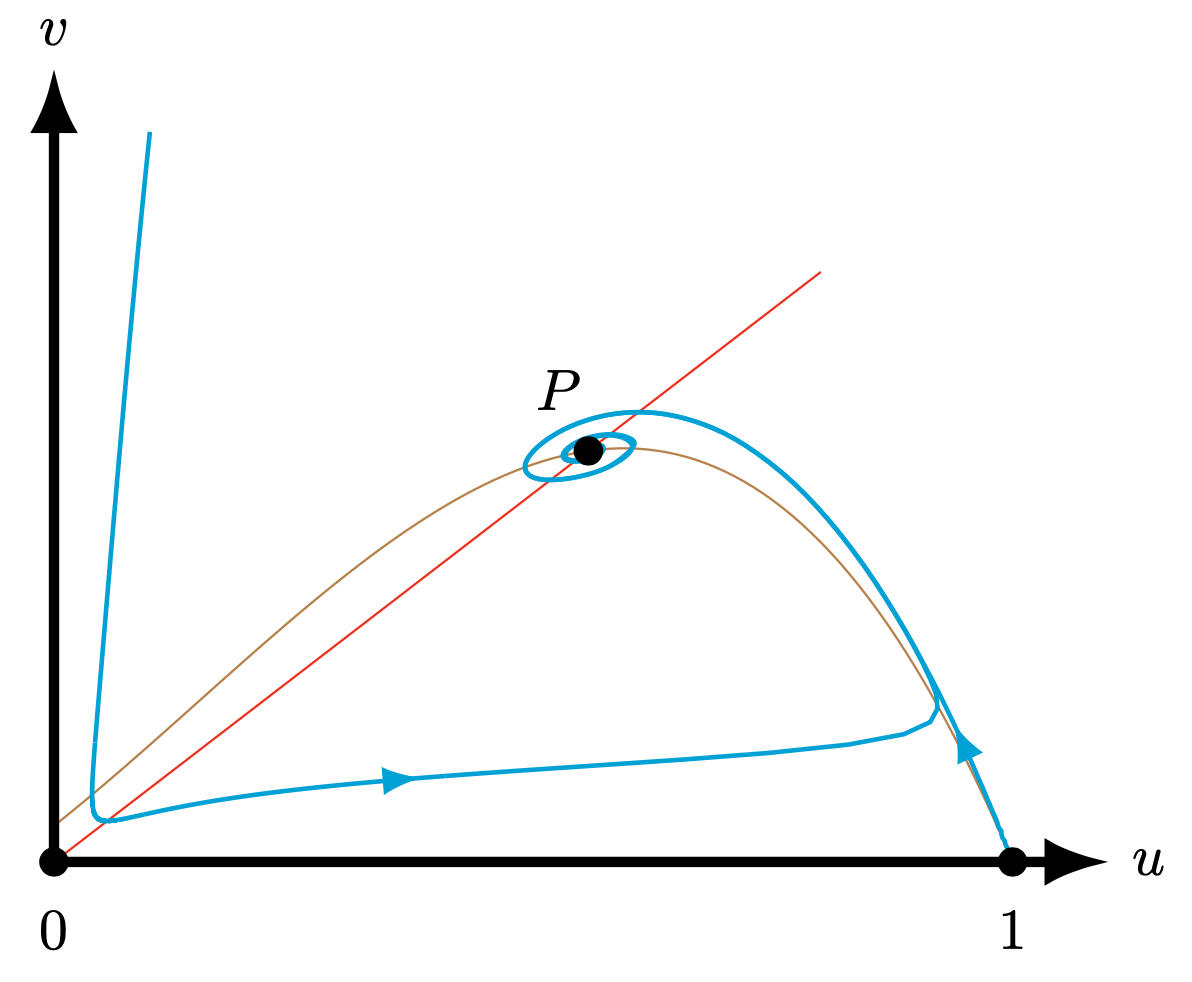}
\caption{If $S=0.1$, then $P$ is an attractor.\\~}\label{F04_a}
\end{subfigure}
\hfill
\begin{subfigure}[b]{0.48\textwidth}\centering
\includegraphics[width=8cm]{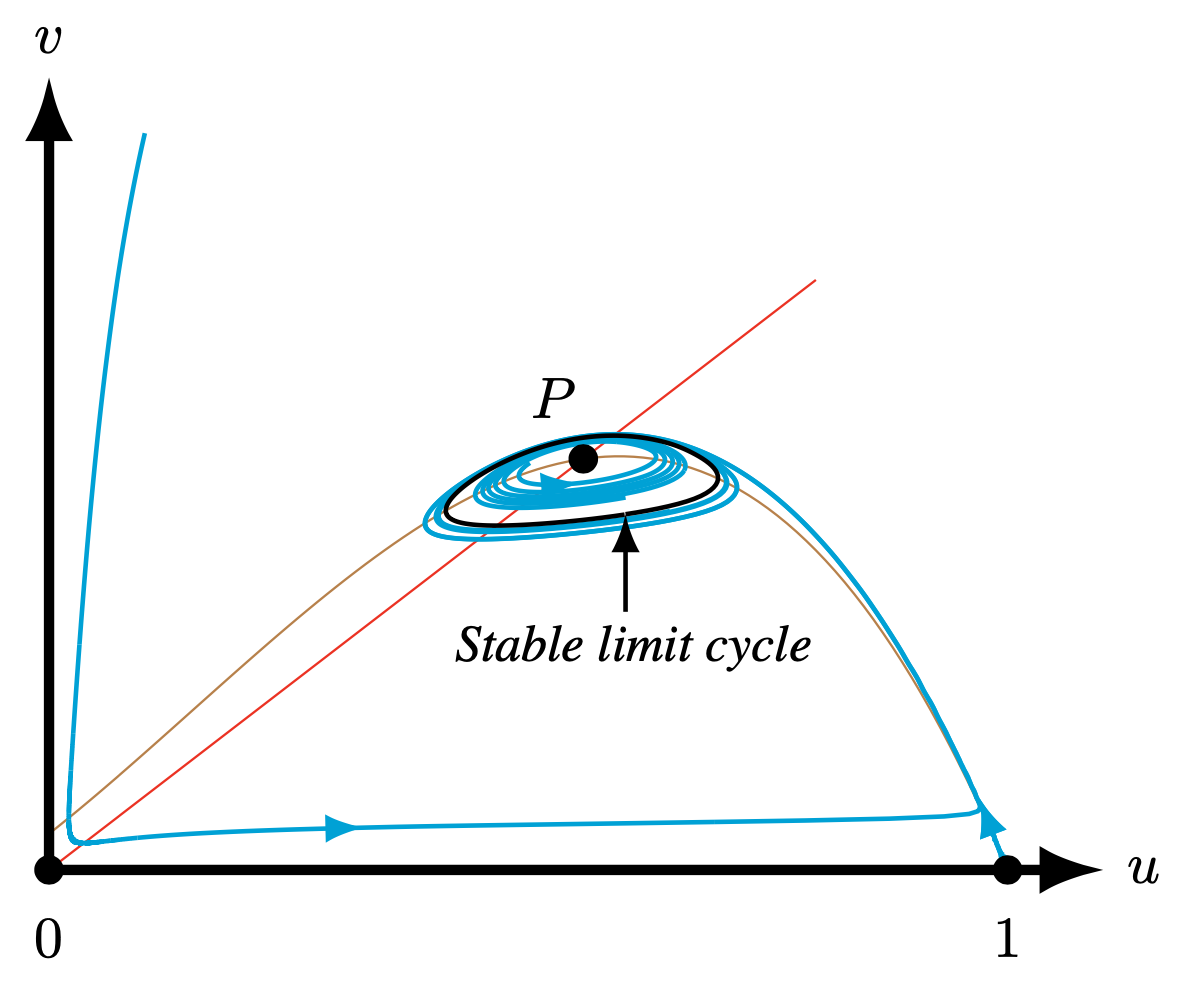}
\caption{If $S=0.045$, then $P$ is a repeller surrounded by a stable limit cycle.}\label{F04_b}
\end{subfigure}
\caption{For $A=0.5$, $M=-0.05$ and $Q=0.51$, such that $T(0.5,-0.05)>0$ and $L(0.5,-0.05,0.51)>0$, i.e. case \ref{NEP1} of Lemma \ref{NEP}, system~\eqref{eq05} has one positive equilibrium $P$. This equilibrium can be an attractor (left panel) or a repeller (right panel), see Corollary~\ref{C01}. In the latter case, the equilibrium is necessarily surrounded by a stable limit cycle. The brown (red) curve represents the predator (prey) nullcline.}
\label{F04}
\end{figure}
\begin{proof} 
This result follows directly from the observation that $\Phi$ \eqref{phi} forms a bounding box and the fact that
the equilibria $\left(0,0\right)$ and $\left(1,0\right)$ are (non-hyperbolic) saddle points, see Lemma~\ref{TEST}.
\end{proof}
Examples of Corollary~\ref{C01} are shown in Figure~\ref{F04}. 

In the case that there are three distinct positive equilibria, Lemma~\ref{L01new} simplifies to the following.
\begin{corr}\label{C02}
Let the system parameters of~\eqref{eq05} be such that the conditions of case~\ref{NEP2}\ref{NEP2b} of Lemma~\ref{NEP} are met. Then, system~\eqref{eq05} has 
three positive equilibria $P_1=(u^*,u^*), P_2=(u^*_-,u^*_-)$ and $P_3=(u^*_+,u^*_+)$. 
If the three equilibria are distinct, then the middle equilibrium is a saddle point,  while the outer two equilibria are a repeller or an attractor. If both outer equilibria are repellers, then there is a stable limit cycle surrounding the equilibria. \end{corr}
\begin{proof}
Recall that we assumed, without loss of generality, $0<u^*<u^*_-<u^*_+<1$, that is, $P_2$ is the middle equilibrium. Now the first two results immediately follow from Lemma~\ref{L01new}, the conditions of case~\ref{NEP2}\ref{NEP2b} of Lemma~\ref{NEP} and the expressions for $u^*_\pm$ in terms of $u^*$ \eqref{eq08}. In particular, for the middle equilibrium $P_2=(u^*_-,u^*_-)$ we have
$$(u_-^*)^2\left(2u_-^*-T\left(A,M\right)\right)-AM=\dfrac{\sqrt{\Delta}}{4}\left(-1-A-M+3u^*+\sqrt{\Delta}\right)
=u_-^*\sqrt{\Delta}\left(u^*-u^*_-\right) <0\,.
$$
Similarly, for $P_3=(u^*_+,u^*_+)$ we get
$$
(u^*_+)^2\left(2u^*_+-T\left(A,M\right)\right)-AM=\dfrac{\sqrt{\Delta}}{4}\left(-T\left(A,M\right)+3u^*-\sqrt{\Delta}\right)
=u^*_+\sqrt{\Delta}\left(u^*_+-u^*\right)>0. 
$$

To show that $P_1$ cannot be a saddle point, we recall that the number of equilibria as discussed in Lemma~\ref{NEP}, and their stability, is independent of our initially picked root $u^*$. Therefore, assume that initially picked root $\tilde{u}^*$ in Lemma~\ref{NEP} is the largest root (where we added the \textasciitilde \, to make the distinction with the previous choice). Now, $0<\tilde{u}^*_-<\tilde{u}^*_+<\tilde{u}^*<1$ and we have that 
$$(\tilde{u}_-^*)^2\left(2\tilde{u}_-^*-T\left(A,M\right)\right)-AM=\dfrac{\sqrt{\Delta}}{4}\left(-1-A-M+3\tilde{u}^*+\sqrt{\Delta}\right)
=\tilde{u}_-^*\sqrt{\Delta}\left(\tilde{u}^*-\tilde{u}^*_-\right) >0\,.
$$
Hence, also $P_1$ cannot be a saddle point.

The final statement again follows directly from the observation that $\Phi$ \eqref{phi} forms a bounding box and the fact that
the equilibria $\left(0,0\right)$ and $\left(1,0\right)$ are (non-hyperbolic) saddle points, see Lemma~\ref{TEST}.
\end{proof}

Examples of Corollary~\ref{C02} are shown in Figure~\ref{F05}. 

\begin{figure} 
\centering
\begin{subfigure}[b]{0.48\textwidth}\centering
\includegraphics[width=8.25cm]{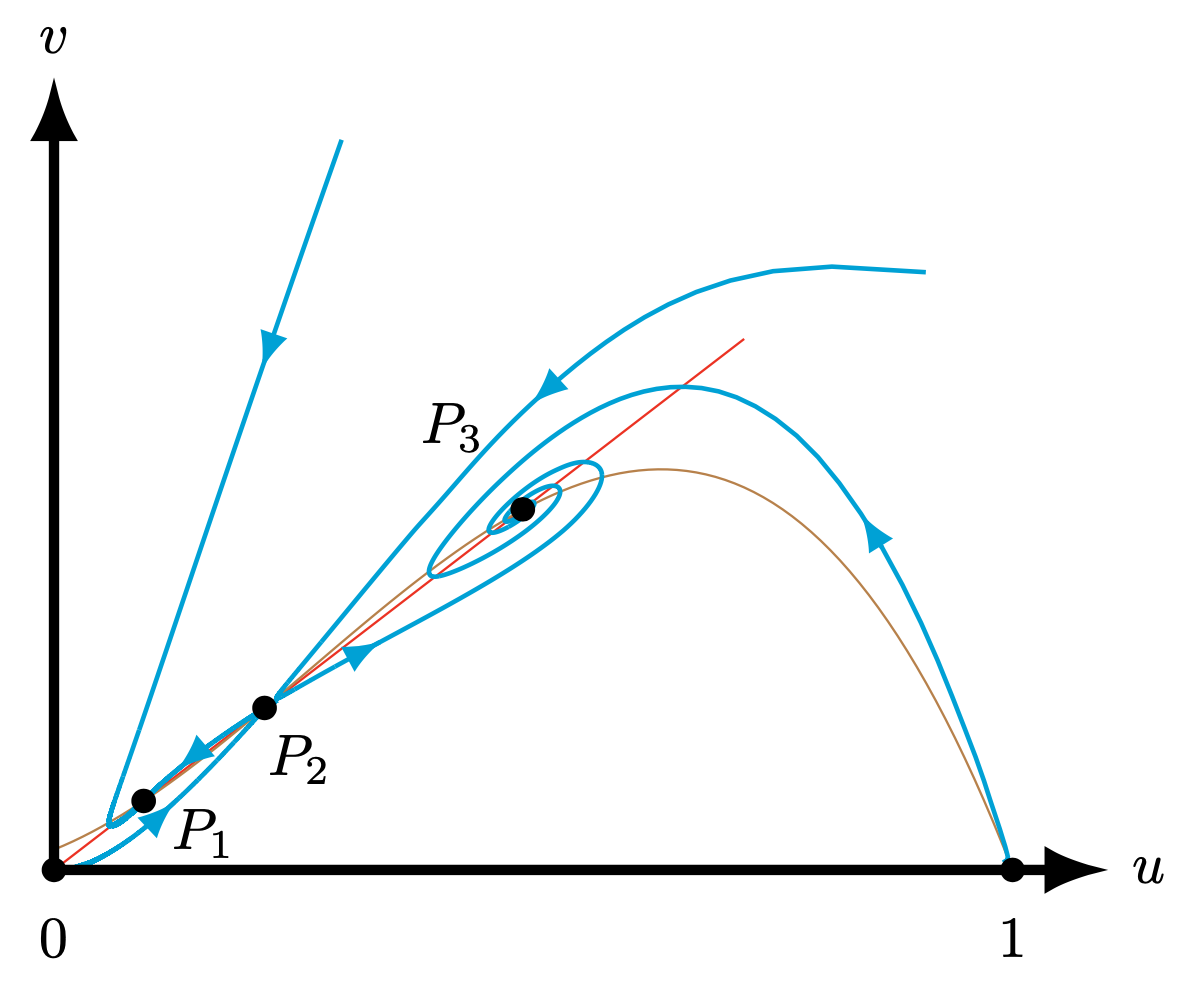}
\caption{For $S=0.3$ the equilibria $P_{1,3}$ are attractors.\\~}\label{F05_a}
\end{subfigure}
\hfill
\begin{subfigure}[b]{0.48\textwidth}\centering
\includegraphics[width=8.25cm]{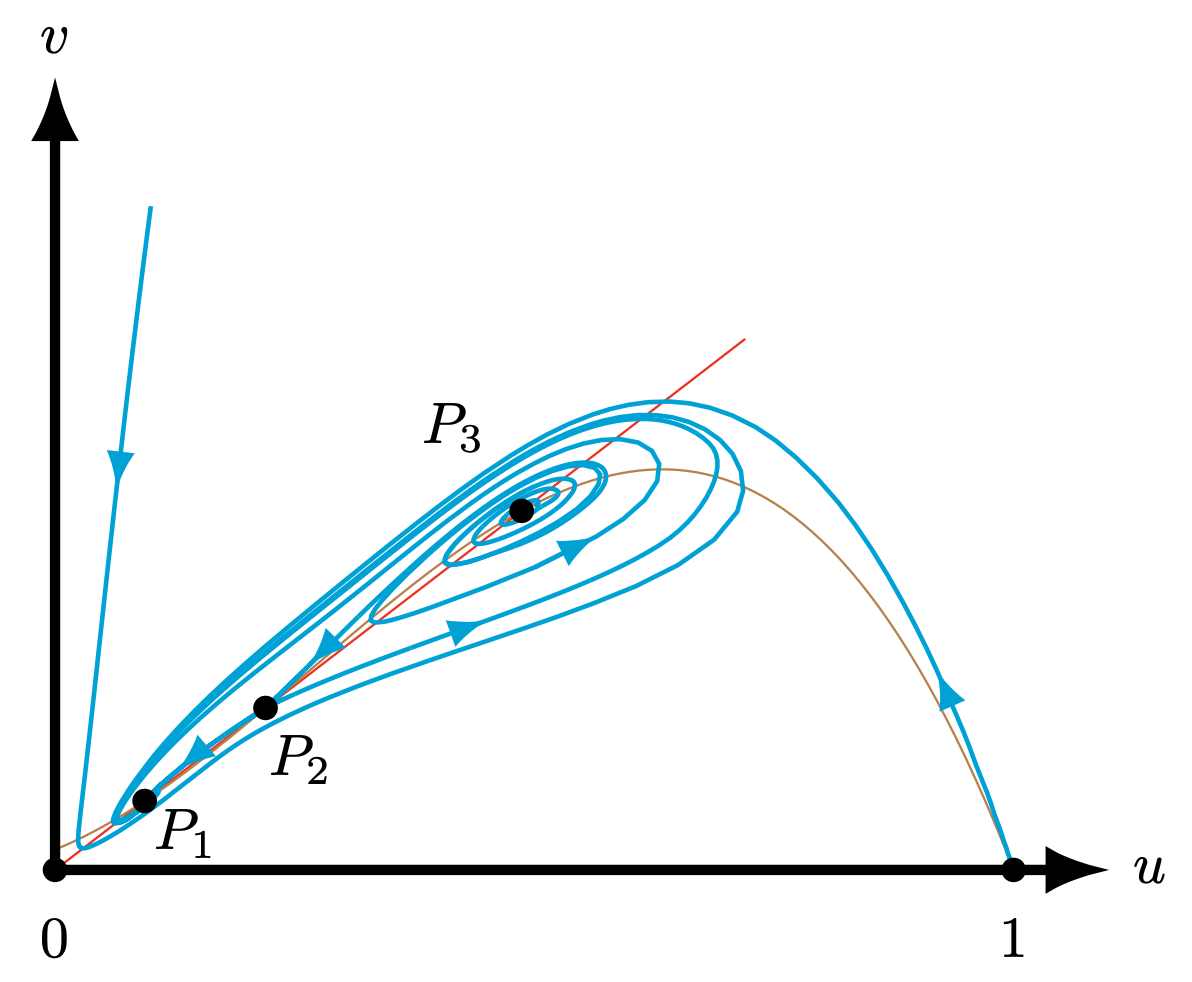}
\caption{For $S=0.2$ the equilibria $P_1$ is an attractor and $P_3$ is a repeller.\\}\label{F05_b}
\end{subfigure}
\hfill
\begin{subfigure}[b]{0.48\textwidth}\centering
\includegraphics[width=8.25cm]{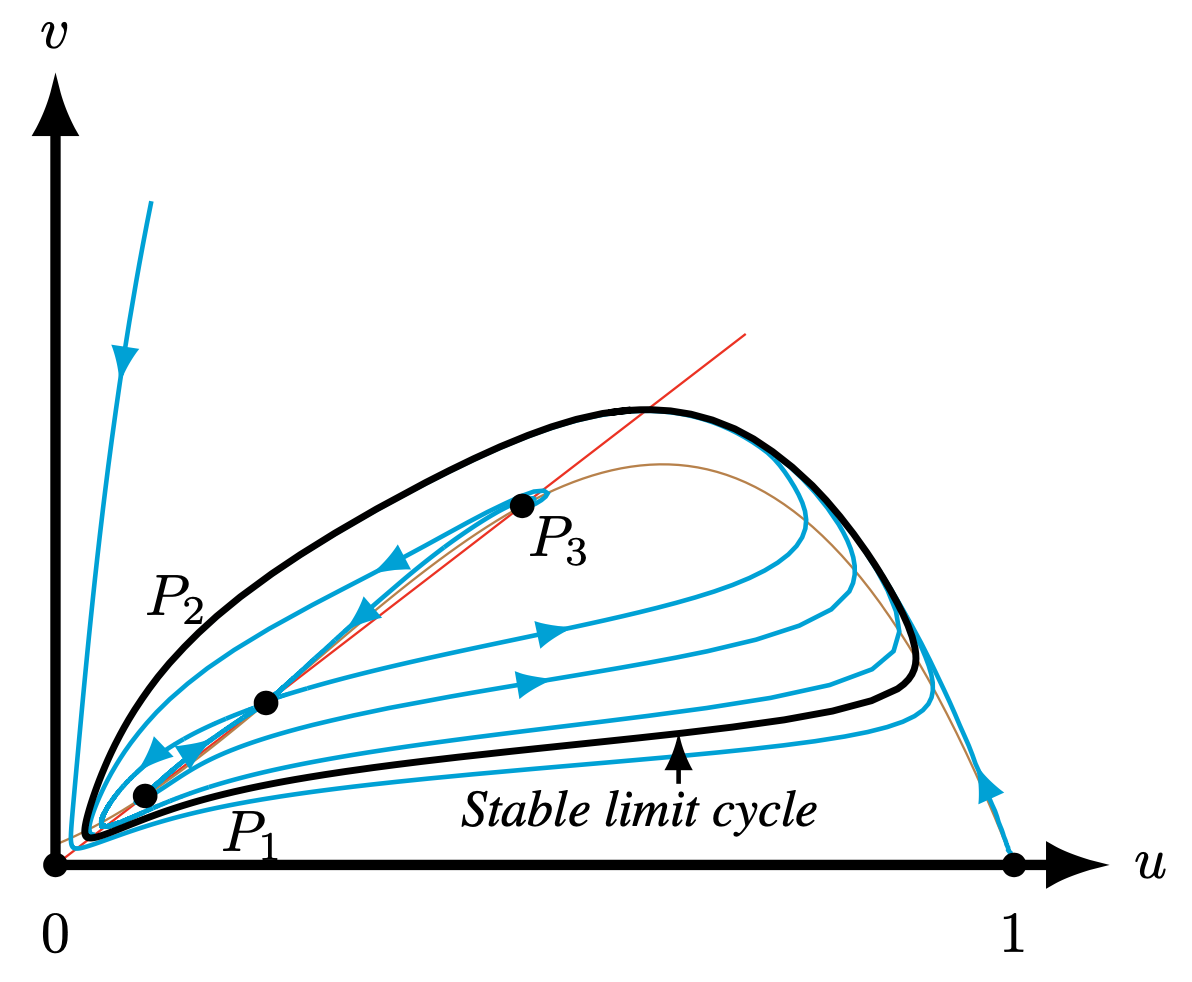}
\caption{For $S=0.13$ the equilibria $P_{1,3}$ are repellers. The equilibria are surrounded by a stable limit cycle.}\label{F05_c}
\end{subfigure}
\caption{Phase planes of system~\eqref{eq05} for $A=0.1$, $M=-0.1$ and $Q=0.363$, such that system~\eqref{eq05} has three positive equilibria $P$, for varying $S$ (Corollary~\ref{C02}). The equilibria $(0,0)$, $(1,0)$ and $P_2$ are always saddles. The brown (red) curve represents the predator (prey) nullcline.}
\label{F05}
\end{figure}

To conclude, we discuss the cases when two of the equilibria collapse, see Figure~\ref{F02} and~\ref{F06}. Since we assumed, without loss of generality, that $0<u^*\leq u^*_-\leq u^*_+<1$, either $P_1 = (u^*,u^*)$ collapses with $P_2=(u^*_-,u^*_-)$ or $P_2$ collapses with $P_3=(u^*_+,u^*_+)$. We obtain the following result for the colliding equilibria.
\begin{corr}\label{C03}
Let the system parameters of~\eqref{eq05} be such that the conditions of case~\ref{NEP2}\ref{NEP2b} of Lemma~\ref{NEP} are met.
\begin{enumerate}[label=(\Roman*)]
\item If $P_1$ collides with $P_2$, then this equilibrium point $P_1=P_2$ of multiplicity two is
	\begin{enumerate}[label=(\roman*)]
	\item \label{C03a} an unstable saddle-node if $S<\dfrac{Qu^*}{A+u^*}$,
	\item \label{C03b} a stable saddle-node if $S>\dfrac{Qu^*}{A+u^*}$.
	\end{enumerate}
\item If $P_2$ collides with $P_3$, then this equilibrium point $P_2=P_3$ of multiplicity two is
	\begin{enumerate}[label=(\roman*)]
	\item \label{C03a} an unstable saddle-node if $S<\dfrac{Q\left(T\left(A,M\right)-u^*\right)}{1+A+M-u^*}$,
	\item \label{C03b} a stable saddle-node if $S>\dfrac{Q\left(T\left(A,M\right)-u^*\right)}{1+A+M-u^*}$.
	\end{enumerate}
		\end{enumerate} 
		\end{corr}
\begin{proof}
We focus on the case where $P_2$ collapses with $P_3$. The proof of the other case goes in a similar fashion and will be omitted. The equilibrium coincide when $\Delta=0$~\eqref{delta} and $u_-^*= u_+^* = \left(T\left(A,M\right)-u^*\right)/2$ . The Jacobian matrix~\eqref{eq13} reduces to 
	\begin{equation}\label{eq15}
	J\left(u_-^*,u_-^*\right)=\begin{pmatrix}
	Q(u_-^*)^2  & -Q(u_-^*)^2 \\ 
	Su_-^*\left( A+u_-^*\right)  &  -Su_-^*\left(A+u_-^*\right) 
	\end{pmatrix}.
	\end{equation}  
	So, as expected, $\det\left(J\left(u_-^*,u_-^*\right)\right)=0$ and the behaviour of the equilibrium depends on the value of the trace  
	$$\tr\left(J\left(u_-^*,u_-^*\right)\right) = Q(u_-^*)^2 -Su_-^*\left(A+u_-^*\right) 
	 \,. $$
Implementing $u_-^*= \left(T\left(A,M\right)-u^*\right)/2$ now gives the desired result. Note that since $u^*_- = u_+^* = \left(T\left(A,M\right)-u^*\right)/2 >0$ implies that $T(A,M)>u^*$. Consequently, $1+A+M-u^*> T(A,M)-u^*>0$. 
\end{proof}

\begin{figure} 
\centering
\begin{subfigure}[b]{0.48\textwidth}\centering
\includegraphics[width=8cm]{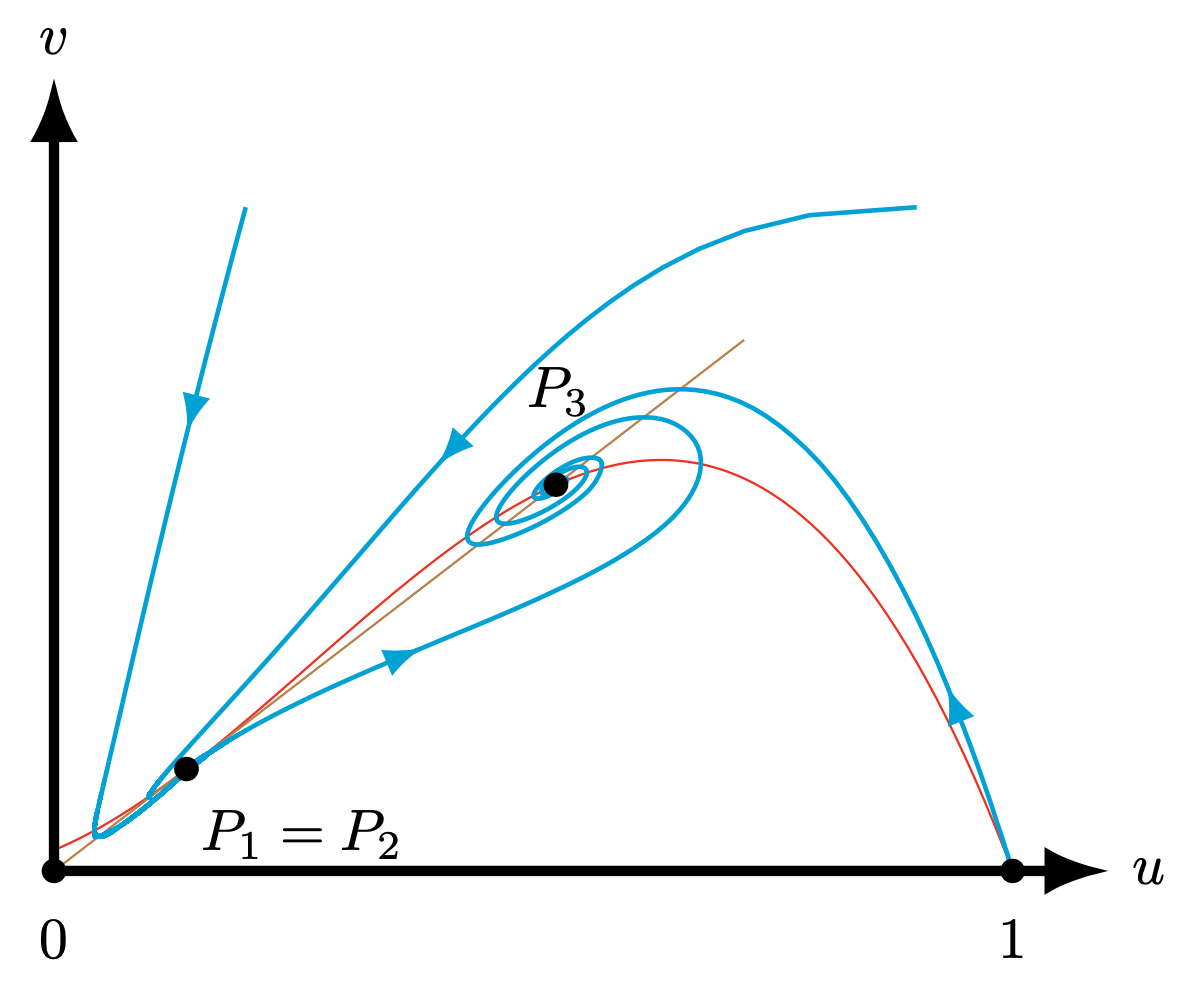}
\caption{For $Q=Q^-(0.1,-0.1) = (73-\sqrt5)/200$ the equilibria $P_1$ and $P_2$ coincide and form a stable saddle-node (Corollary~\ref{C03}\ref{C03a}), while the equilibrium $P_3$ is an attractor (Lemma~\ref{L01new}\ref{L01newb}).}\label{F06_a}
\end{subfigure}
\hfill
\begin{subfigure}[b]{0.48\textwidth}\centering
\includegraphics[width=8cm]{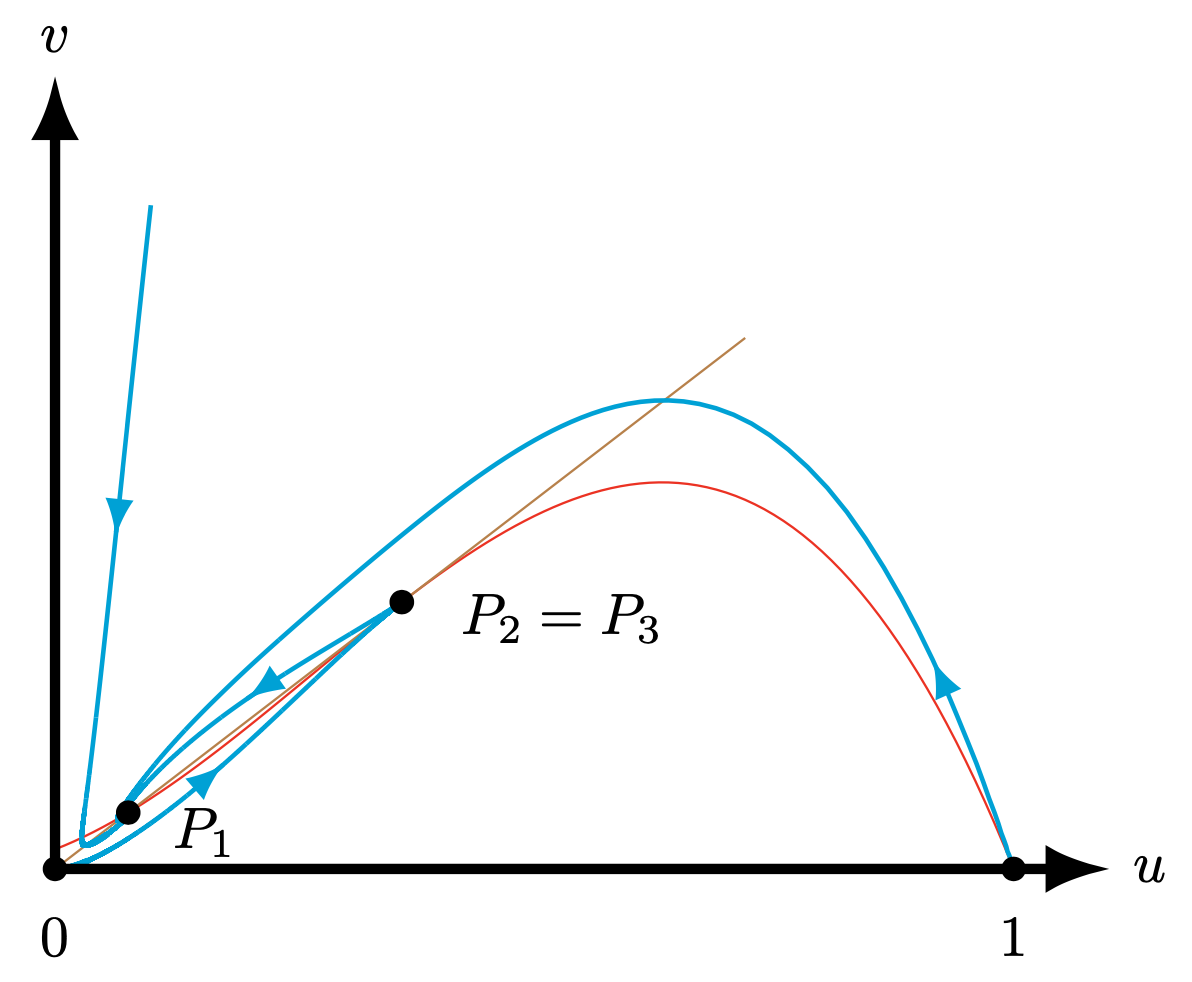}
\caption{For $Q=Q^+(0.1,-0.1) = (73+\sqrt5)/200$ the equilibria $P_2$ and $P_3$ coincide and form an unstable saddle-node (Corollary~\ref{C03}\ref{C03a}), while the equilibrium $P_1$ is an attractor (Lemma~\ref{L01new}\ref{L01newb}).}\label{F06_c}
\end{subfigure}
\caption{Phase planes of system~\eqref{eq05} for $A=0.1$, $M=-0.1$ and $S=0.25$ for varying $Q$. The equilibria $(0,0)$ and $(1,0)$ are always saddles. The brown (red) curve represents the predator (prey) nullcline.}
\label{F06}
\end{figure}

\begin{figure} 
\centering
\begin{subfigure}[b]{0.48\textwidth}\centering
\includegraphics[width=7cm]{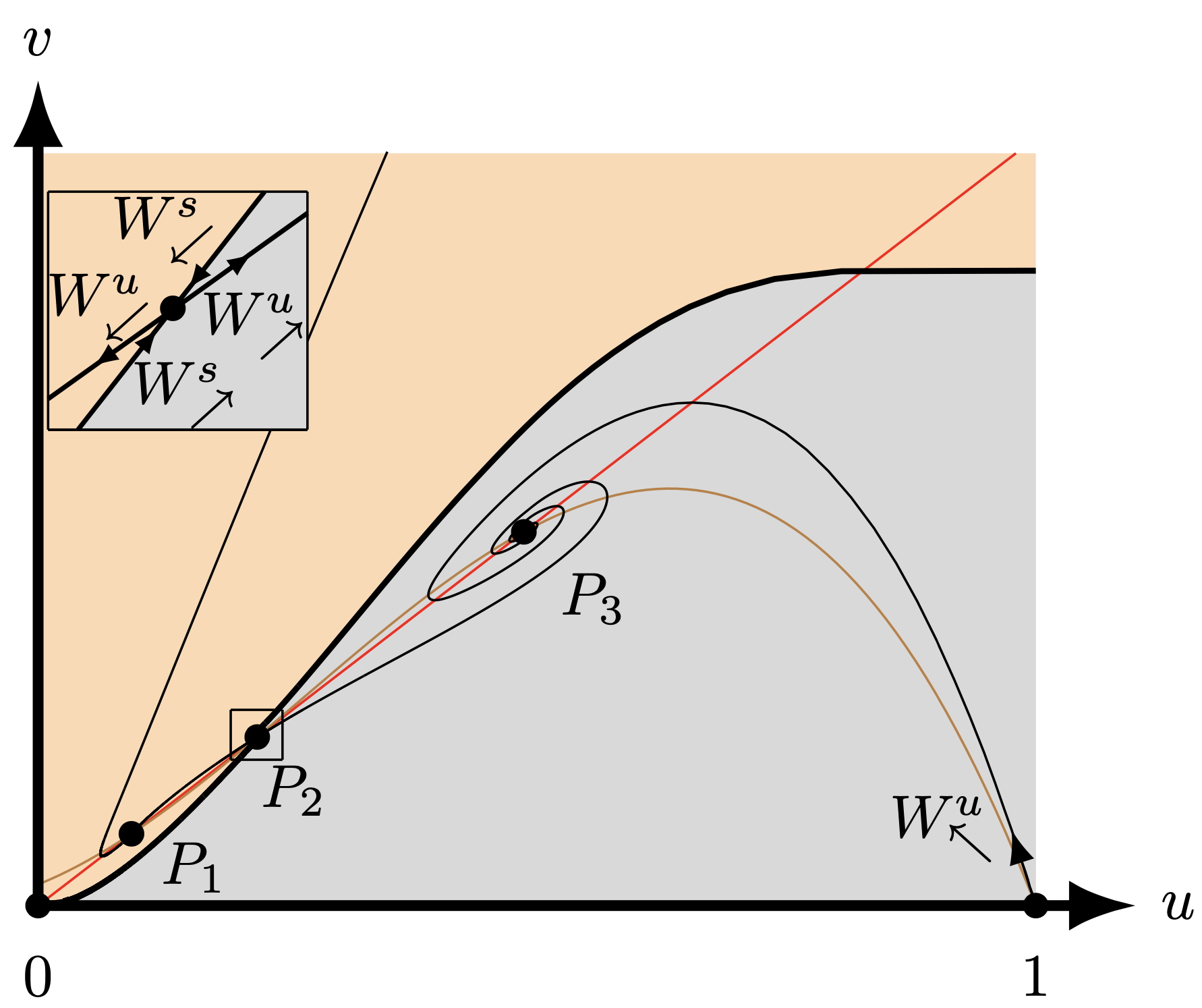}
\caption{For $S=0.3$ both $P_1$ and $P_3$ are attractors. The stable manifold of $P_2$ ($W^s$) forms a separatrix between the basin of attraction of $P_1$ (orange region) and the basin of attraction of $P_3$ (gray region) and $W^u_\nwarrow\left(1,0\right)$ connects to $P_3$.}\label{F09_a}
\end{subfigure}
\hfill
\begin{subfigure}[b]{0.48\textwidth}\centering
\includegraphics[width=7cm]{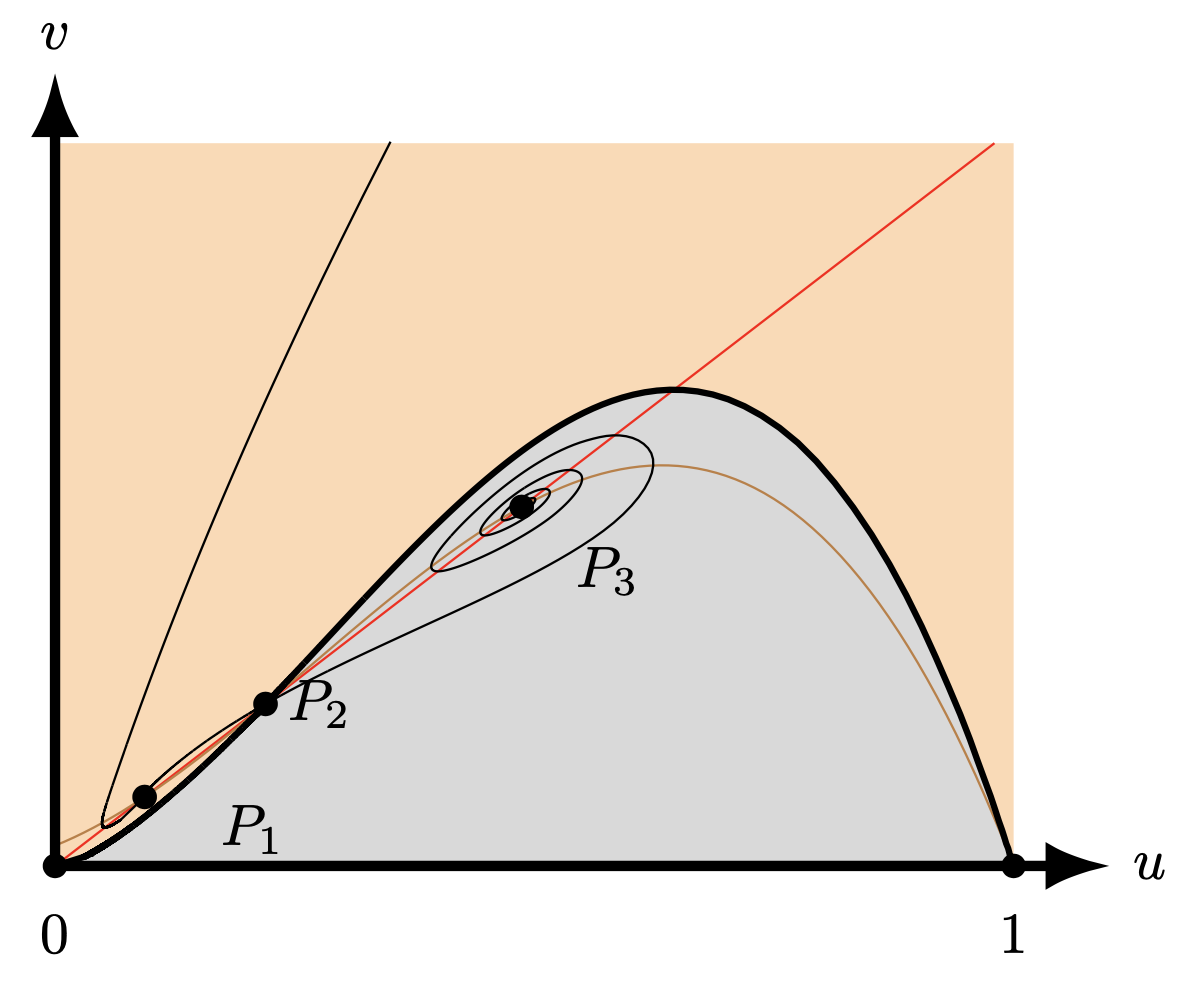}
\caption{For $S=0.24962827$ both $P_1$ and $P_3$ are still attractors and the stable manifold of $P_2$ still forms a separatrix between the two basins of attraction. However, the separatrix now connects to $W^u_\nwarrow\left(1,0\right)$ forming a heteroclinic curve between $P_2$ and $(1,0)$.}\label{F09_b}
\end{subfigure}
\hfill
\begin{subfigure}[b]{0.48\textwidth}\centering
\includegraphics[width=7cm]{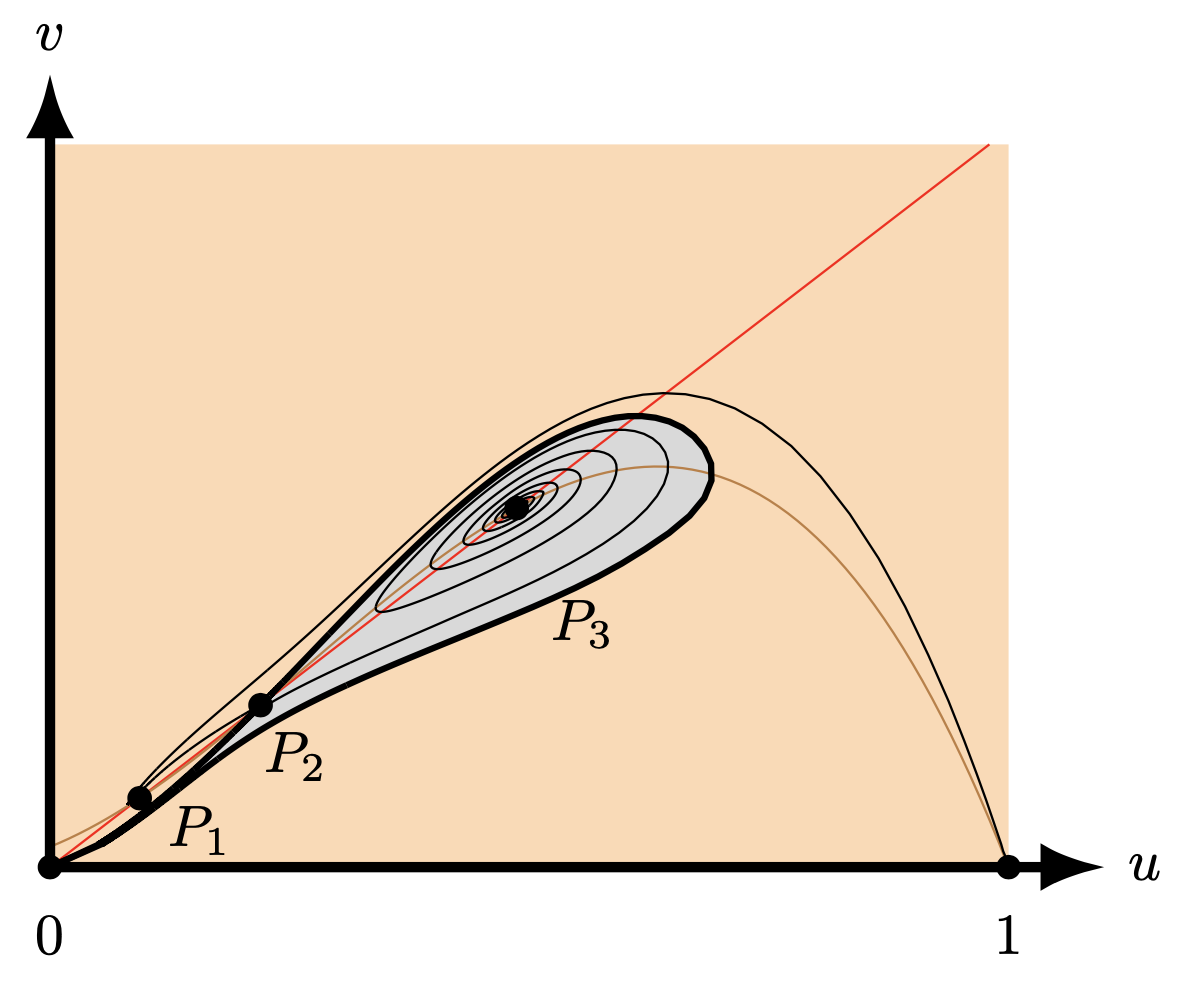}
\caption{For $S=0.235$ both $P_1$ and $P_3$ are attractors. The stable manifold of $P_2$ forms a separatrix between the basin of attraction of $P_1$ (orange region) and the basin of attraction of $P_3$ (gray region). Here, $W^s_\nearrow\left(P_2\right)$ connects to $(0,0)$ and $W^u_\nwarrow\left(1,0\right)$ connects to $P_1$.}\label{F09_c}
\end{subfigure}
\hfill
\begin{subfigure}[b]{0.48\textwidth} \label{F09_b}\centering
\includegraphics[width=7cm]{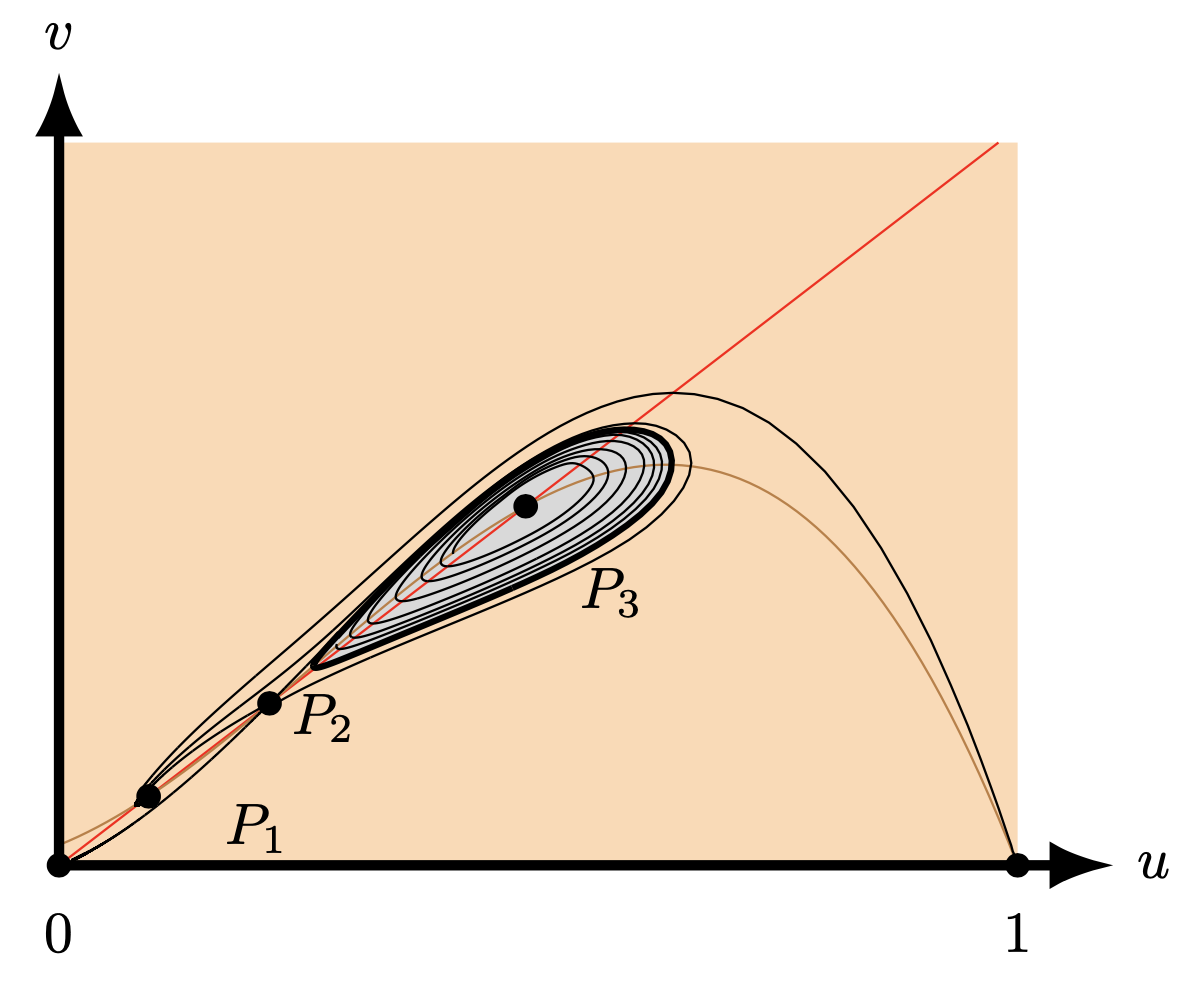}
\caption{For $S=0.225$ both $P_1$ and $P_3$ are still attractors and $P_3$ is surrounded by an unstable limit cycle. This limit cycle forms a separatrix between the two basins of attraction. Here, $W^u_\nwarrow\left(1,0\right)$ connects to $P_1$}\label{F09_d}
\end{subfigure}
\hfill
\begin{subfigure}[b]{0.48\textwidth}\centering
\includegraphics[width=7cm]{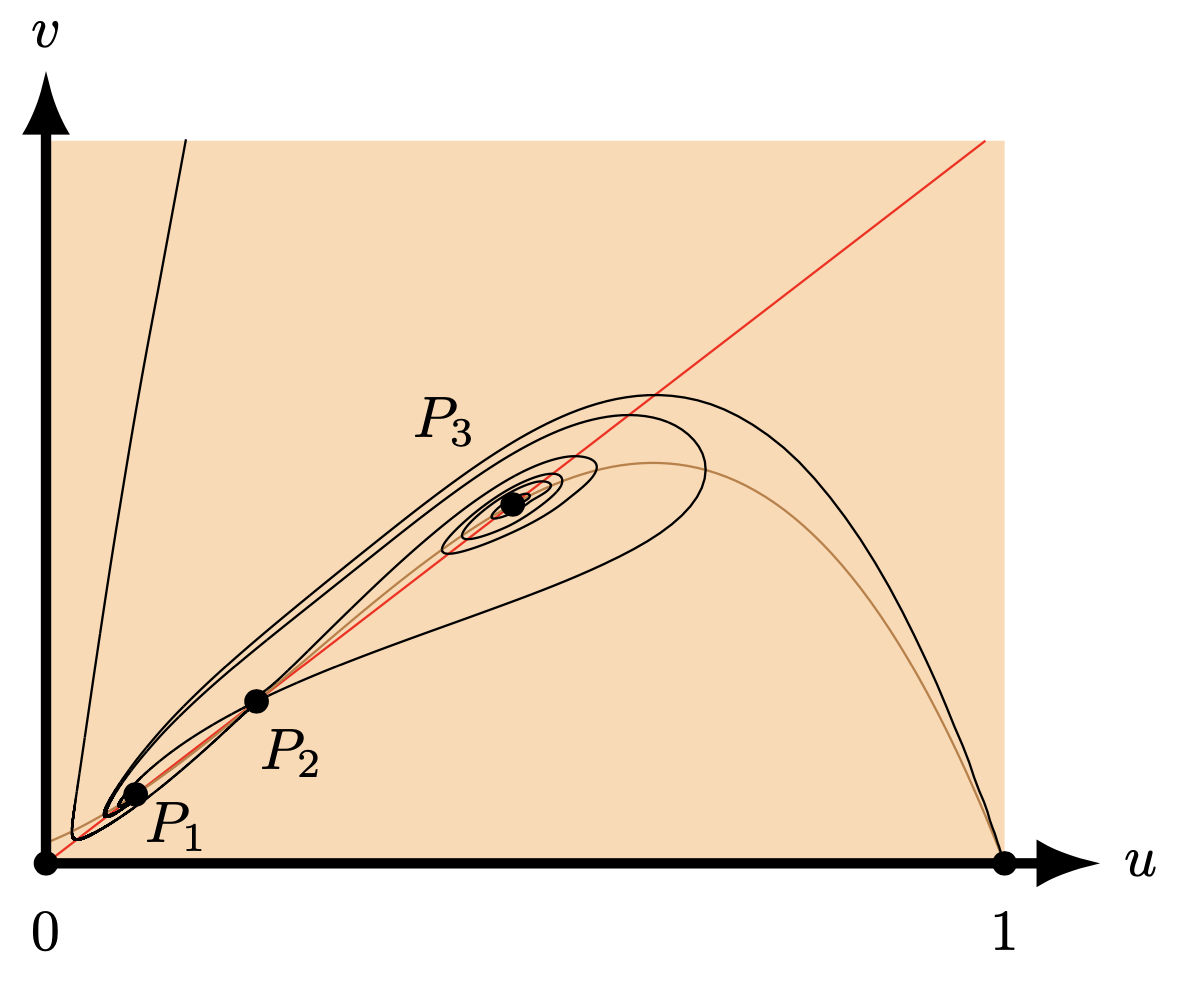}
\caption{For $S=0.18$ the equilibria $P_3$ is a repeller and $P_1$ is a global attractor. Hence, also $W^u_\nwarrow\left(1,0\right)$ now connects to $P_1$.\\~}\label{F09_e}
\end{subfigure}
\hfill
\begin{subfigure}[b]{0.48\textwidth}\centering
\includegraphics[width=7cm]{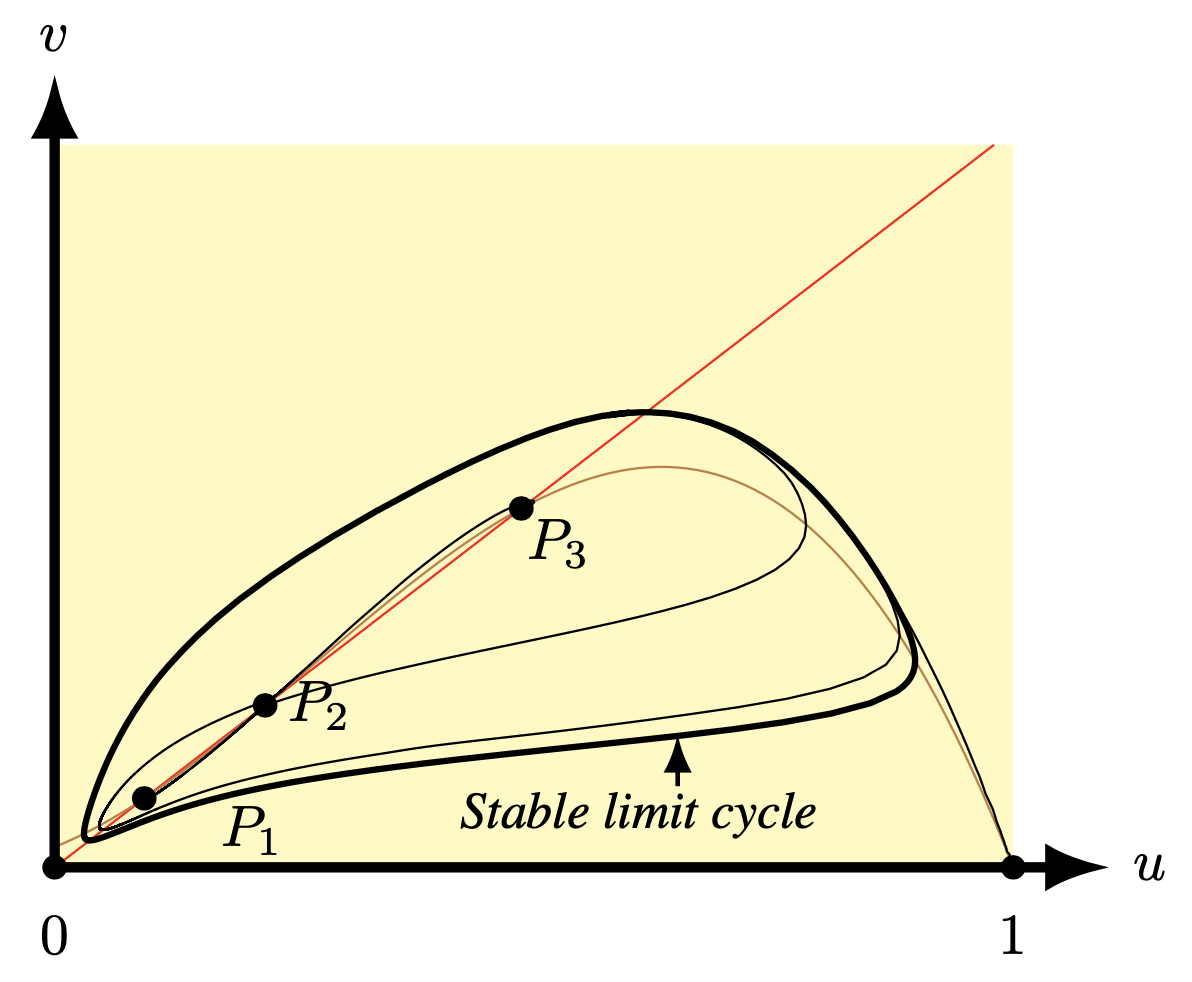}
\caption{For $S=0.13$ both $P_1$ and $P_3$ are repellers and $W^u_\nwarrow\left(1,0\right)$ connects to a stable limit cycle which surrounds the three positive equilibria.}\label{F09_f}
\end{subfigure}
\hfill
\caption{Phase plane of system~\eqref{eq05} for $A=0.1$, $M=-0.1$ and $Q=0.363$ for varying $S$. The equilibria $(0,0)$, $(1,0)$ and $P_2$ are always saddles. The brown (red) curve represents the predator (prey) nullcline.}
\label{F09}
\end{figure}
Since equation~\eqref{eq06} does not depend on the parameter $S$ it does not affect the number of positive equilibria. In contrast, modifying $Q$ impacts $\Delta$ \eqref{delta} and hence the number of positive equilibria. If we assume that the system parameters are such that we have three distinct positive equilibria $P_1=(u^*,u^*), P_2=(u^*_-,u^*_-)$ and $P_3=(u^*_+,u^*_+)$ with $u^*<u^*_-<u^*_+$, then the equilibrium $P_2$ is a saddle point, while the other equilibria $P_1$ and $P_3$ are repellers or attractors, see Corollary~\ref{C02}. Let $W^s_\swarrow (P_2)$ denote the stable manifold of $P_2$ that approaches $P_2$ from a northeast direction and  $W^s_\nearrow (P_2)$ the stable manifold of $P_2$ that approaches $P_2$ from a southwest direction, see panel~(a) of Figure~\ref{F09}. By continuity of the vector field in $S$ and the parameter $Q$ fixed (see Figure~\ref{F09}), the following behaviour is observed for the equilibria $P_1$ and $P_3$:
\begin{enumerate}[label=(\roman*)]
\item For large $S$, Lemma~\ref{L01new} yields that both $P_1$ and $P_3$ are repellers (i.e. for $S> \max_{\tilde{u} \in (0,1)} f(\tilde{u})$, where $f$ is defined in Lemma~\ref{L01new}) and consequently we observe that $W^s_\swarrow(P_2)$ connects to the boundary of region $\Phi$~\eqref{phi}. In particular, $W^s(P_2)=W^s_\nearrow (P_2)\bigcup W^s_\swarrow (P_2)$ creates a separatrix curve between the basin of attraction of $P_1$ and $P_3$. Hence, any solutions having initial conditions above the separatrix lie in the basin of attraction of $P_1$, see orange region of Figure~\ref{F09}. Whereas any solutions with initial conditions under of the separatrix lie in the basin of attraction of $P_3$, see gray region of Figure~\ref{F09}. The $\alpha$-limit of $W^s_\swarrow (P_2)$ is outside of $\Phi$, hence the curve $W^s_\swarrow (P_2)$ lies above $W^u_\nwarrow\left(1,0\right)$, the unstable manifold of $(1,0)$ that leaves $(1,0)$ in a northwest direction and necessarily remains in $\Phi$. Hence $W^u_\nwarrow\left(1,0\right)$ connects to $P_3$, see panel~(a) of Figure~\ref{F09}\footnote{{\it A priori}, the role of $P_1$ and $P_3$ could be interchanged}. 
\item Then, by further reducing $S$, there exists conditions in the $\left(Q,S\right)$-parameters space for which the two manifolds $W^u_\nwarrow\left(1,0\right)$ and $W^s_\swarrow (P_2)$ coincide, forming the heteroclinic curve~\cite{chicone}, see panel~(b) of Figure~\ref{F09}. 
\item Upon further decreasing $S$, $W^s(P_2)$ connects with $(0,0)$, see panel~(c) of Figure~\ref{F09}. 
Furthermore, there exists an $S$-value for which $W^s_{\swarrow}(P_2)$ connects with $W^u_{\nearrow}(P_2)$ (i.e. $W^s_{\swarrow}(P_2) = W^u_{\nearrow}(P_2)$) generating an homoclinic curve. Note that this case is not shown in Figure~\ref{F09} but it lies between $S=0.235$ of panel (c) and $S = 0.225$ of panel (d). \item When the homoclinic breaks an unstable limit cycle is created around $P_2$, see panel~(d) of Figure~\ref{F09}. 
\item Upon further decrasing $S$, $P_3$ becomes unstable and $P_1$ turns into an attractor and $W^u_\nwarrow\left(1,0\right)$ connects to $P_1$, see panel~(e) of Figure~\ref{F09}. 
\item Finally, if the system parameters are such that $f(u^*)$ and $f(u^*_+)$ are positive, where $f$ is defined in Lemma~\ref{L01new}, then, $P_1$ and $P_3$ are repellers for $S< \min\{f(u^*),f(u^*_+)\}$. Hence, by Corollary~\ref{C02} the system possesses a limit cycle, and $W^u_\nwarrow\left(1,0\right)$ connects to this limit cycle, while $W^s_\swarrow (P_2)$ connects with $W^u_\swarrow (P_3)$ and $W^s_\nearrow (P_2)$ connects with $W^u_\nearrow (P_1)$, see panel~(f) in Figure~\ref{F09}. 
\end{enumerate}
Note that the dynamics described above is also observed in studies related to heteroclinic cycles. These dynamics refer to heteroclinic trajectories which are saddle-type of invariant sets. For instance, in~\cite{labouriau} the authors observed that the system could experience by continuity a saddle-node bifurcation, a Hopf bifurcation and a homoclinic bifurcation. 

\subsection{Bifurcation Analysis}\label{BA}
In this section we discuss some of the potential bifurcation scenarios. For brevity, we focus only on the case where $P_2$ collapses with $P_3$. To note, similar results, but under different parameter conditions, hold for the other case (i.e. $P_1$ collides with $P_2$)  and the proofs, which will be omitted, go in a similar fashion.

\begin{theo}\label{SN}
Let the system parameters of~\eqref{eq05} be such that the conditions of case~\ref{NEP2}\ref{NEP2b} of Lemma~\ref{NEP} are met and assume that $P_2$ and $P_3$ coincide, i.e. $\Delta=0$~\eqref{delta}. 
Then, system~\eqref{eq05} experiences a saddle-node bifurcation at the equilibrium $P_2=P_3=\left(u^*_-,u^*_-\right)$.
\end{theo}
\begin{proof}
For $\Delta=0$ the equilibria $P_2$ and $P_3$ collapse and reduce to $\left(u^*_-,u^*_-\right)=\left(T\left(A,M\right)-u^*\right)/2$ \eqref{eq08}. The other positive equilibrium is $P_1 = (u^*,u^*)$ with $u^*< u^*_-$ by assumption.
So, the Jacobian matrix of the system~\eqref{eq05} evaluated at the equilibrium $\left(u^*_-,u^*_-\right)$ is 
	\[\begin{aligned}
	J\left(u^*_-,u^*_-\right) & =\begin{pmatrix}
	\dfrac{Q\left(u^*-T\left(A,M\right)\right)^2}{4} & -\dfrac{Q\left(u^*-T\left(A,M\right)\right)^2}{4} \\ 
	& \\
	\dfrac{S\left(1+A+M-u^*\right)\left(u^*-T\left(A,M\right)\right)}{4} & -\dfrac{S\left(1+A+M-u^*\right)\left(u^*-T\left(A,M\right)\right)}{4}
	\end{pmatrix}\\
	\end{aligned}\,,\]
	see also \eqref{eq15}, and $\det\left(J\left(u^*_-,u^*_-\right)\right)=0$. Let 
	$V=(v_1, v_2)^T = (1, 1)^T$ be 
	the eigenvector corresponding to the eigenvalue $\lambda=0$ of the matrix $J\left(u^*_-,u^*_-\right)$. Additionally, let $U=(u_1 ,u_2 )^T=(S\left(1+A+M-u^*\right)/(Q\left(u^*-T\left(A,M\right)\right)),  1)^T$.

	The dynamical system~\eqref{eq05} in vector form is given by
	\begin{equation} \label{eq16}
	f\left(u,v;Q\right) =\begin{pmatrix}
	\left(u+A\right)\left(1-u\right)\left(u-M\right)-Qv\\ 
	u-v
	\end{pmatrix}\,.
	\end{equation}
	Differentiating $f(u,v,Q)$
	with respect to the bifurcation parameter $Q$ and evaluating at $P_2$ gives
	\[f_Q\left(u^*_-,u^*_-,Q\right)=\begin{pmatrix}
	\dfrac{u^*-T\left(A,M\right)}{2}\\[2mm]
	0
	\end{pmatrix}.\]
	Therefore,
	\[U\cdot f_Q\left(u,v;Q\right)=\dfrac{S\left(1+A+M-u^*\right)}{2Q}\neq0.\]
	Next, we analyse the expression $U\cdot D^2f\left(u,v;Q\right)\left(V,V\right)$. The latter is given by  
	\[\begin{aligned}
	D^2f\left(u,v;Q\right)\left(V,V\right) & =\dfrac{\partial^2f\left(u,v;Q\right)}{\partial u^2}v_1v_1+\dfrac{\partial^2f\left(u,v;Q\right)}{\partial u\partial v}v_1v_2 +\dfrac{\partial^2f\left(u,v;Q\right)}{\partial v\partial u}v_2v_1+\dfrac{\partial^2f\left(u,v;Q\right)}{\partial v^2}v_2v_2\\
	& = \begin{pmatrix}
	-2\left(2+A-M\right)\\ 
	0
	\end{pmatrix}. 
	\end{aligned}\]
	Thus,
	\[\begin{aligned}
	U\cdot D^2f\left(u,v;Q\right)\left(V,V\right)=-\dfrac{2S\left(2+A-M\right)\left(1+A+M-u^*\right)}{Q\left(u^*-T\left(A,M\right)\right)}\neq0
	\end{aligned}.\]
	Therefore, by Sotomayor's Theorem~\cite[Section 4.2, e.g]{perko} system~\eqref{eq05} has a saddle-node bifurcation at $P_2=\left(u^*_-,u^*_-\right)$.
\end{proof}

\begin{theo}\label{BT}
Let the system parameters of~\eqref{eq05} be such that the conditions of case~\ref{NEP2}\ref{NEP2b} of Lemma~\ref{NEP} are met and assume that $P_2$ and $P_3$ coincide, i.e. $\Delta=0$~\eqref{delta}, and let
\begin{equation} \label{SS}S=\dfrac{Q\left(u^*-T\left(A,M\right)\right)}{\left(1+A+M-u^*\right)}.\end{equation}
Then, the equilibrium point $P_2=P_3=(u^*_-,u^*_-)$ is a cusp point.
\end{theo}
See~\cite{andronov}, for a formal definition of a cusp point.
\begin{proof}
	If $S=Q\left(u^*-T\left(A,M\right)\right)/\left(1+A+M-u^*\right)$, then $\det\left(J\left(u^*_-,u^*_-\right)\right)=0$ and $\tr\left(J\left(u^*_-,u^*_-\right)\right)=0$ and the Jacobian matrix of system~\eqref{eq05} evaluated at the equilibrium $\left(u^*_-,u^*_-\right)$ simplifies to
	\[J\left(u^*_-,u^*_-\right) =-\dfrac{1}{4}S\left(1+A+M-u^*\right)\left(u^*-T\left(A,M\right)\right)\begin{pmatrix}
	1 & -1 \\ 
	1 & -1 
	\end{pmatrix}.\]
	Now, we find the Jordan normal form for $J\left(u^*_-,u^*_-\right)$. It has repeated eigenvalues and a unique eigenvector
	$(1,1)^T$. This vector will be the first column of the matrix of transformations $\Upsilon_4$. To obtain the second column, we choose a vector that makes the matrix $\Upsilon_4$ non-singular. In this case, we take $(-1,0)^T$. 
Thus,
	\[\begin{aligned}
	\Upsilon_4 &= \begin{pmatrix}
	1 & -1 \\ 
	1 & 0 
	\end{pmatrix}, 
	\end{aligned}\]
	and
	\[\begin{aligned}
	\Upsilon_4^{-1}\left(J\left(u^*_-,u^*_-\right)\right)\Upsilon_4 &=\begin{pmatrix}
	0 & -\dfrac{1}{4}S\left(1+A+M-u^*\right)\left(u^*-T\left(A,M\right)\right) \\ 
	0 & 0 
	\end{pmatrix} . 
	\end{aligned}\]
	Hence, we have that the equilibrium $\left(u^*_-,u^*_-\right)$ is a codimension 2 cusp point~\cite{xiao2} for $(Q,S)$ such that $\Delta=0$. 
\end{proof}
A Bogdanov–Takens bifurcation can be obtained by following~\cite{xiao2,huang,zhu} where the authors showed that this type of bifurcation can be obtained by unfolding the system around the cusp of codimension two. In other words, if we consider a small neighbourhood of the equilibrium point $P_2=P_3=(u^*_-,u^*_-)$ and $Q$ and $S$ vary near $S_0\left(Q_0\right):=Q_0\left(u^*-T\left(A,M\right)\right)/\left(1+A+M-u^*\right)$, then system~\eqref{eq05} undergoes a Bogdanov--Takens bifurcation.
\begin{figure}
\begin{center}
\includegraphics[width=10.5cm]{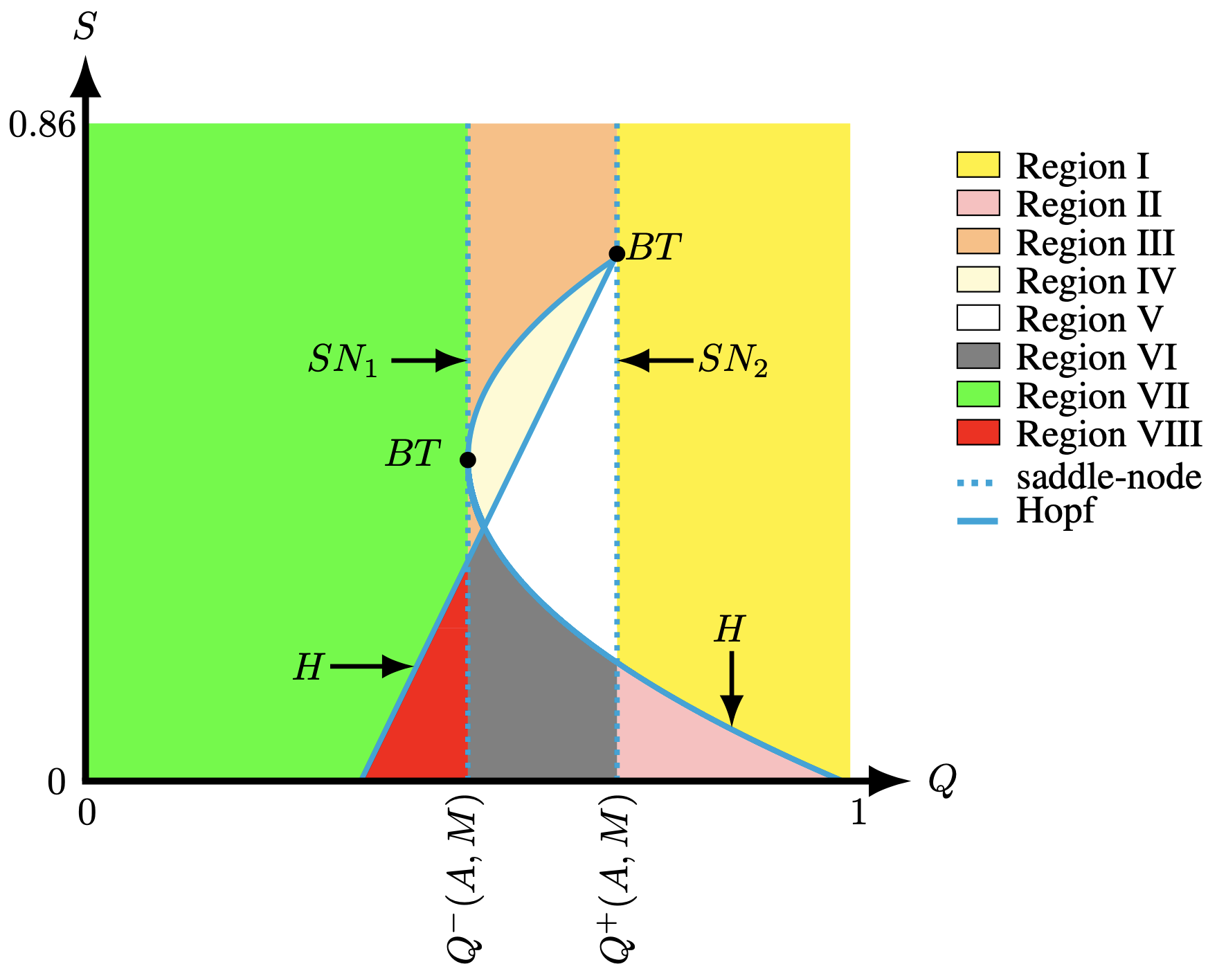}
\caption{The bifurcation diagram of system~\eqref{eq05} for $\left(A,M\right)=\left(0.1,-0.1\right)$ fixed created with the numerical bifurcation package MATCONT~\cite{matcont}. The curve $H$ represents the Hopf curve, $SN_{1,2}$ represents the saddle-node curve, and $BT$ represents the Bogdanov--Takens bifurcation.}\label{F08}
\end{center}
\end{figure}
Luckily, nowadays, there are several computational methods to find Bogdanov--Takens points and these methods are implemented in software packages such as MATCONT~\cite{matcont}. Figure~\ref{F08} illustrates the two Bogdanov--Takens bifurcations which were detected with MATCONT in the $\left(Q,S\right)$-plane for $\left(A,M\right)$ fixed. In particular, the bifurcation curves obtained from Theorem~\ref{SN} and~\ref{BT} divide the $\left(Q,S\right)$-parameter-space into nine regions, see Figure \ref{F08}. From our results we observe that for $A$ and $M$ fixed modifying the parameter $Q$ impacts the number of positive equilibria of system \eqref{eq05}. In contrast, the modification of the parameter $S$ only changes the stability of the positive equilibria $P_1$ and $P_3$ of system \eqref{eq05}, while the other equilibria $\left(0,0\right)$, $\left(1,0\right)$ and $P_2$ are always saddle points. When the parameters lie in the curve $Q=Q^-(A,M)$ the equilibria $P_1$ and $P_2$ collapse and we have $P_1=P_2=(u^*,u^*)$ and $P_3=(u^*_+,u^*_+)$ with $u^*<u^*_+$. In addition, when parameters lie in the curve $Q=Q^{+}(A,M)$ the equilibria $P_2$ and $P_3$ collapse and  we have $P_1 = (u^*,u^*)$ and $P_2=P_3= (\left(T\left(A,M\right)-u^*\right)/2,\left(T\left(A,M\right)-u^*\right)/2)$. Along these lines we observe a saddle-node bifurcation, see Theorem~\ref{SN}. The system experiences a Bogdanov--Takens bifurcation if, in addition, $S=Q\left(u^*-T\left(A,M\right)\right)/\left(1+A+M-u^*\right)$, see Theorem~\ref{BT}. When the parameters are located in Region I and VII, system~\eqref{eq05} has one positive equilibrium which is an attractor, while when the parameters are located in Region II and VIII system~\eqref{eq05} also has one positive equilibrium but now it is a repeller surrounded by a stable limit cycle, see Figure~\ref{F04}. 
When the parameters moved to Regions III--VI system~\eqref{eq05} has three equilibria. In these regions, $P_2$ is always a saddle point. 
In region III $P_1$ and $P_3$ are both attractors, see panels~\ref{F09_a}--\ref{F09_c} in Figure~\ref{F09}. 
Furthermore, when the parameters lie in Regions IV $P_1$ is an attractor and $P_3$ is an attractor surrounded by an unstable limit cycle, see panel~\ref{F09_d} in Figure~\ref{F09}.  When the parameters lie in Regions V $P_1$ is an attractor and $P_3$ is a repeller, see~\ref{F09_e} in Figure~\ref{F09}. Finally, when the parameters are located in Region VI, $P_1$ and $P_3$ are both repellers and the equilibria $P_{1,2,3}$ are thus surrounded by a stable limit cycle, see~\ref{F09_f} in Figure~\ref{F09}. Note that Figure~\ref{F08} only shows a partial bifurcation diagram, since there could exist a homoclinic curve which is located between the Hopf and the saddle-node curve. The homoclinic bifurcation and other types of bifurcations are not include in Figures~\ref{F09} and~\ref{F08}, however it could be observed in the transition from panel~\ref{F09_c} to panel~\ref{F09_d} for $S\in(0.225,0.235)$. In addition, the bifurcation curves observed in Figure~\ref{F08} are qualitatively similar to the bifurcation studied in~\cite{labouriau}. In particular, the authors observed that the variation of two parameters impacts the number of equilibria of the system and the stability of these equilibria. It was also observed that, by continuity, the stability of an equilibrium point changes from stable to unstable point surrounded by a stable limit cycle. Then, the limit cycle increases until it joins a homoclinic curve which is the connection between the stable and unstable manifold of a saddle point.

\section{Conclusions}\label{con}
In this manuscript, we studied a Leslie--Gower predator-prey model with weak Allee effect and functional response Holling type II, i.e. system~\eqref{eq03} with $m<0$. We simplified the analysis by studying a topologically equivalent system~\eqref{eq05}. The topologically equivalent system~\eqref{eq05} has two equilibria on the axis which are always (non-hyperbolic) saddle points, see Lemma~\ref{TEST}. 
In addition, system~\eqref{eq05} has at most three positive equilibria in the first quadrant, see Lemma~\ref{NEP} and Figure~\ref{F02}.  
As the function $\varphi$ is a diffeomorphism preserving the orientation of time, the dynamics of system \eqref{eq03} is topologically equivalent to system \eqref{eq05} \cite[Theorem 1]{arancibia3}. Therefore, we can, for instance, conclude from the results of Section~\ref{S02} that there are conditions on the system parameter for which the predator and prey can coexist without oscillations or with oscillations, see Figure~\ref{F09} and this behaviour depends intrinsically and nonlinearly on the system parameters including the predation rate ($q$) and the intrinsic growth rate of the predator ($s$). 

In more detail, when the system parameters are such that system~\eqref{eq05} has only one positive equilibrium, i.e. $Q < Q^-(A,M)$ or $Q> Q^+(A,M)$, see Regions I, II, VII and VIII in Figure~\ref{F08}, then it is a repeller surrounded by a stable limit cycle or an attractor, but not a saddle point, see Corollary~\ref{C01} and Figure~\ref{F04}. 
Note that when the equilibrium is an attractor it is not necessary a global attractor since there is a small region in parameter space near the Hopf bifurcation where the equilibrium is surrounded by two limit cycles, see panel (a) of Figure~\ref{F10}. 
The observed behaviour for system~\eqref{eq05}, and hence system~\eqref{eq03} with $m<0$, in this case of one positive equilibrium -- global attractor, attractor with two limit cycles, or repeller with one limit cycle -- is similar to the behaviour of the original Leslie--Gower predator-prey \eqref{eq02} with logistic growth for the prey \cite{saez}. In other words, 
for $q$ smaller than some $q^-$ or larger than some $q^+$ both systems~\eqref{eq02} and~\eqref{eq03} with $m<0$ present qualitatively similar behaviours, see panel~\ref{F04_b} in Figure~\ref{F04}.
\begin{figure} 
\centering
\begin{subfigure}[b]{0.48\textwidth}\centering
\includegraphics[width=8cm]{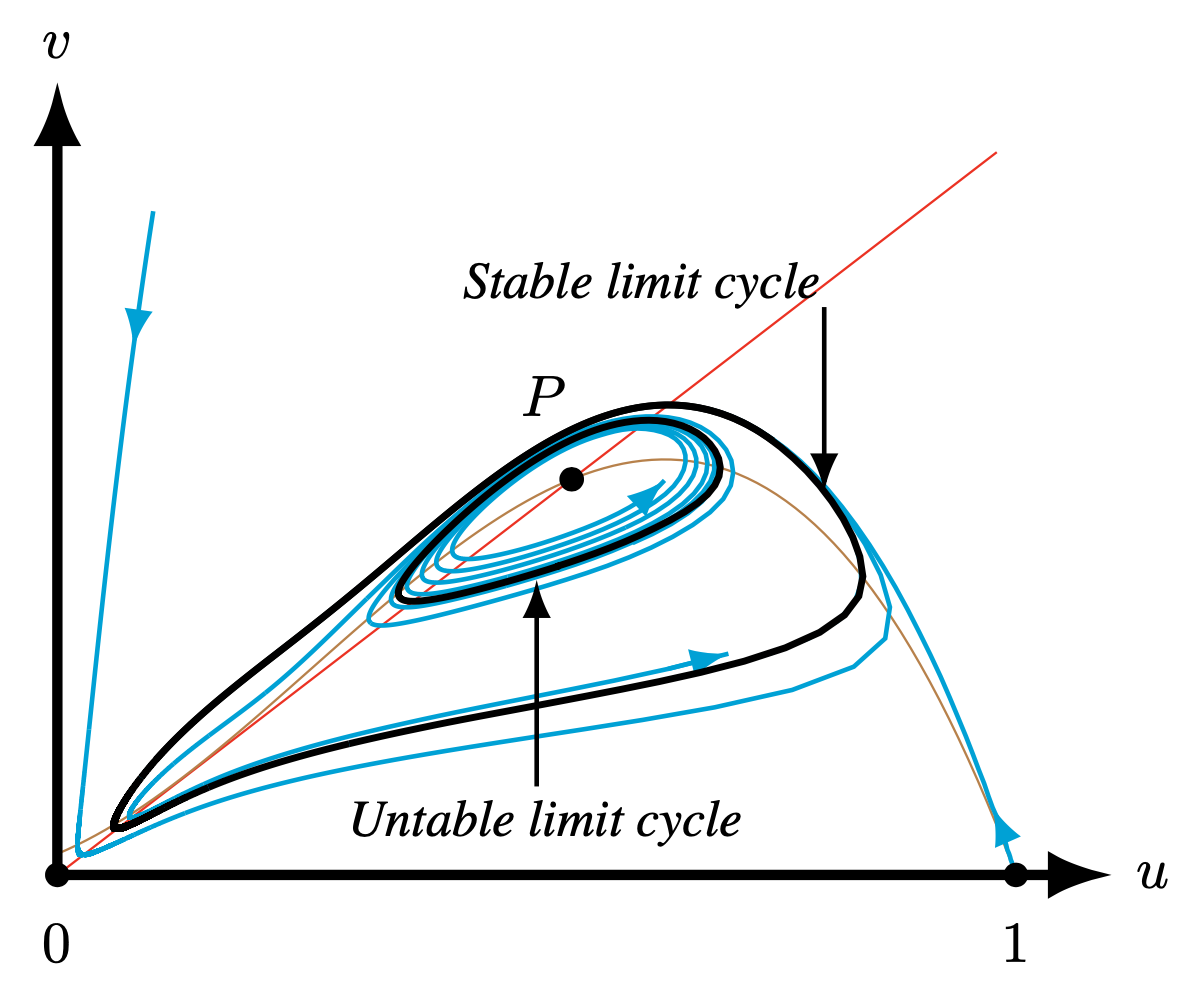}
\caption{For $Q=0.345$ and $S=0.134332$ $P_1$ is an attractor surrounded by two limit cycles, the inner one is unstable, while the outer limit cycle is stable.\\~}\label{F10_a}
\end{subfigure}	
\hfill
\begin{subfigure}[b]{0.48\textwidth}\centering
\includegraphics[width=8cm]{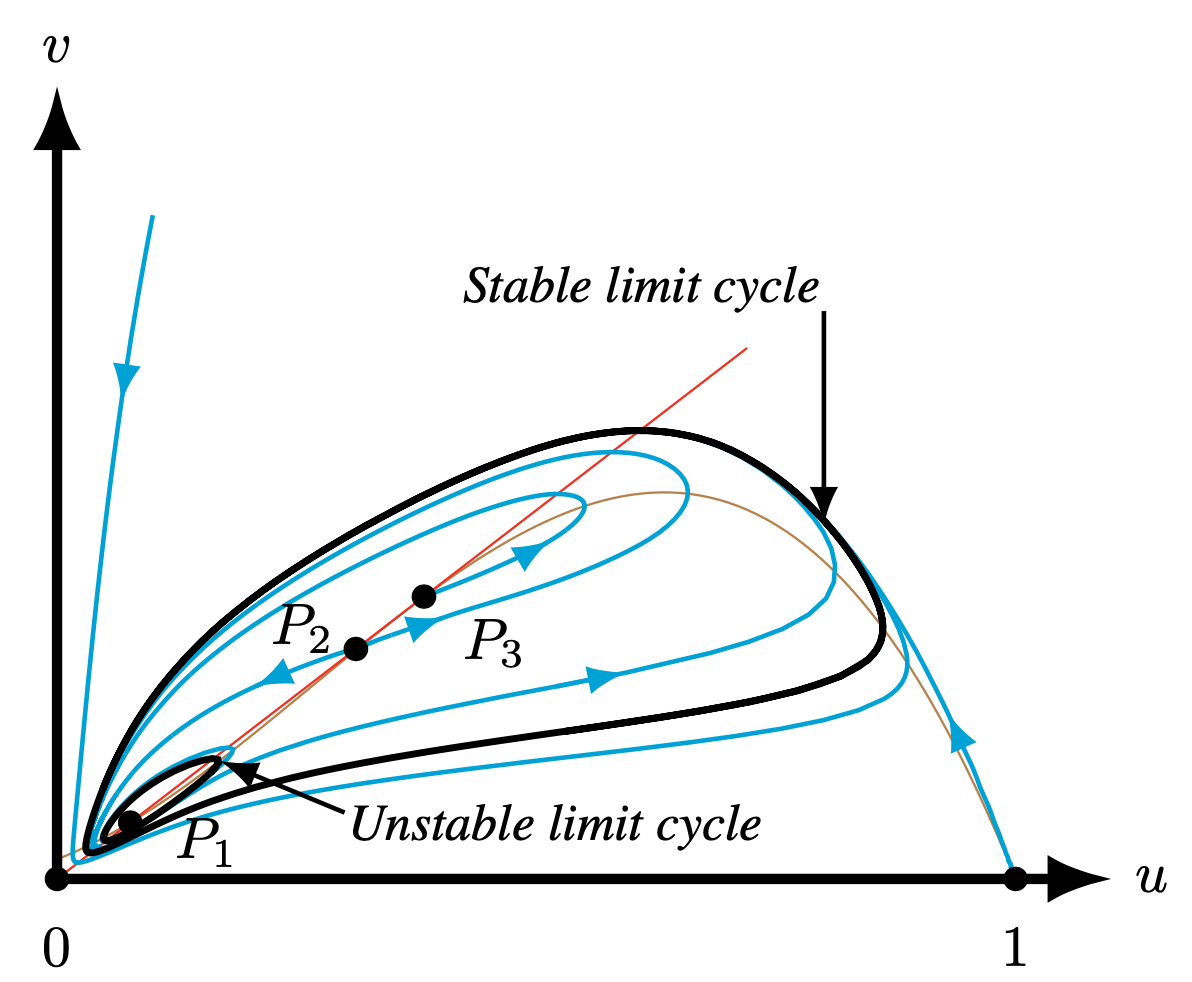}	
\caption{For for $Q=0.363$ and $S=0.1298$ system~\eqref{eq05} has three equilibria. Here, $P_2$ is a saddle, $P_1$ is an attractor surrounded by an unstable limit cycles and $P_3$ is a repeller. The three positive equilibria $P_{1,2,3}$ are surrounded by a stable limit cycle.}\label{F10_b}
\end{subfigure}
\caption{Phase plane of system~\eqref{eq05} for $A=0.1$ and $M=-0.1$ for varying $Q$ and $S$. The equilibria $(0,0)$ and $(1,0)$ are always saddles. The brown (red) curve represents the predator (prey) nullcline.}
\label{F10}
\end{figure}

In contrast, for $Q^-(A,M)< Q < Q^+(A,M)$, system~\eqref{eq05} has three equilibria in the first quadrant, which is not possible in the original Leslie--Gower predator-prey \eqref{eq02} nor in the model with strong Allee effect, i.e. system~\eqref{eq03} with $m>0$.
In this case, the middle one is always a saddle point, while the outer two are repellers or attracters, see Corollary~\ref{C02} and Figures~\ref{F05} and ~\ref{F09}. If the two outer equilibria are attractors, then either the stable manifold of the saddle point determines a separatrix curve which divides the basins of attraction of the two attracters, or there exists an unstable limit cycle dividing the basins of attraction, see panels (a) -- (d) of Figure~\ref{F09}. When both outer equilibria are repellers the system necessarily possesses a stable limit cycle, see panel (f) of Figure~\ref{F09}. 
When the outer equilibria are an attractor and a repeller, the attractor can be a global attractor or there are two limit cycles surrounding the equilibria, see panel (e) of Figure~\ref{F09} and panel (b) of Figure~\ref{F10}. The latter one again happens near the Hopf bifurcation.
In other words, there are regions in parameter space for which system~\eqref{eq05} has two attractors or an attractor and a stable limit cycle. As such, a modification of one, or both, of the species could have the result that you end up in a different basin of attraction and you will thus have significantly different dynamics as you will approach the other attractor, see panels~(a)--(d) on Figure~\ref{F09} and panel (b) of Figure~\ref{F10}.
We also performed an initial (numerical) bifurcation analysis which revealed the existence of saddle-node and Bogdanov--Takens bifurcations, see Theorems~\ref{SN} and \ref{BT} and Figure~\ref{F08}.

In summary, we showed that the weak Allee effect in the Leslie--Gower model~\eqref{eq03} presents -- due to the presence of a region in parameter space where we have three positive equilibria -- richer dynamics than the original Leslie--Gower predator-prey model~\eqref{eq03} studied, for example, by Saez and Gonzalez-Olivares \cite{saez}. The original Leslie--Gower predator-prey model~\eqref{eq03} does not possess the bifurcations studied in Section~\ref{BA}. That is, a saddle-node bifurcation and a Bogdanov–Takens bifurcation. Moreover, the original model only has one positive equilibrium point, which can be stable surrounded by two limit cycles. Therefore, the populations could switch between a stable state or an oscillatory behaviour. However, system~\eqref{eq03} with a weak Allee effect can have two stable equilibria and thus depending on the initial conditionthe population could end up in two stable states. From an ecological point of view, we can also conclude that the species in system~\eqref{eq03} could coexist or oscillate but never go extinct, see Figure~\ref{F09}. This behaviour was also observed by Ostfeld and Canhan~\cite{ostfeld}. The authors found that the stabilisation of vole (\textit{Microtus agrestis}) populations in Southeastern New York depend on the variation in reproductive rate and thus this impact the reproductive activity of the species. This is again similar to the original Leslie--Gower predator-prey model~\eqref{eq03}, but different from the Leslie--Gower model~\eqref{eq03} with strong Allee effect, i.e. $m>0$. In the latter case, the model supports coexistence and extinction of the species~\cite{arancibia3}. 

\bibliographystyle{unsrt}

\end{document}